\documentclass{amsart}
\usepackage{amsmath}
\usepackage{framed,comment,enumerate}
\usepackage[pdftex]{graphicx}
\usepackage{epstopdf}
\usepackage{xcolor, framed}
\usepackage{color}

\def\R {\mathbb{R}}
\def\eps{\varepsilon}
\def\OP{obstacle problem }

\def\Qua{\mathcal{Q}}
\def\UQua{\mathcal{UQ}}

\def\Sph{\mathbb{S}^{d-1}}

\def\SPE{\mathcal{S}(p,\eps,1)}

\def\PosS{\{u>0\}}

\def\Dee{D_{ee}}

\def\hh{\hat{h}}
\def\hu{\hat{u}}
\def\hO{\hat{O}}

\def\minx{\underline{x}'}

\newtheorem{prop}{Proposition}[section]
\newtheorem{thm}{Theorem}[section]

\newtheorem{lem}{Lemma}[section]
\theoremstyle{definition}
\newtheorem{defi}{Definition}[section]
\newtheorem{rem}{Remark}[section]
\numberwithin{equation}{section}

\title[Regularity of the singular set]{Regularity of the singular set in the fully nonlinear obstacle problem } 

\author{Ovidiu Savin}
\address{Department of Mathematics,	Columbia University, New York, USA}
\email{savin@math.columbia.edu}

\author{Hui Yu}
\address{Department of Mathematics,	Columbia University, New York, USA}
\email{ huiyu@math.columbia.edu}
\thanks{O.~S.~is supported by  NSF grant DMS-1800645.}
\thanks{H.~Y.~is supported by NSF grant DMS-1954363.}

\title[Regularity of the singular set]{On the fine regularity of the singular set in the nonlinear obstacle problem } 

\begin{document}

\begin{abstract}
We revisit and sharpen the results from our previous work, where we investigated the regularity of the singular set of the free boundary in the 
nonlinear obstacle problem. As in the work of Figalli-Serra on the classical obstacle problem, we show that each stratum can be further decomposed into a `good' part and an `anomalous' part, where the former is covered by $C^{1,1-}$ manifolds, and the latter is of lower dimension. 
 \end{abstract}

\maketitle

\section{Introduction}

In this paper we revisit and sharpen the results of our previous work \cite{SY}, where we investigated the regularity of the singular set of the free boundary in the nonlinear obstacle problem. 

In its simplest form, the \textit{classical obstacle problem} consists in solving
\begin{equation}\label{eq1}
\triangle u= \chi_{\{u>0\}},
\end{equation}
in the class of nonnegative functions $u \ge 0$ defined in a domain $\Omega \subset \R^d$, subject to boundary conditions. An important feature of 
this problem is the presence of a \textit{free boundary} $\partial \{u>0\}$, which separates the coincidence region $\{u=0\}$ from the positivity set 
$\{u>0\}$. 

The regularity of the free boundary was established by Caffarelli in \cite{C1,C2} through a blow-up analysis. At each point $x_0$ on $\partial\{u>0\}$, the 
quadratic blow-ups of $u$ converge to either a half-quadratic polynomial like $\frac 12 [(x \cdot e_{x_0})^+]^2$ for some $e_{x_0}\in\Sph$, or to a whole-quadratic polynomial like $\frac 12 
x^T A_{x_0} x$ with $A_{x_0} 
\ge 0$ and $trace(A_{x_0})=1$. In the first case, we say that $x_0$ is a \textit{regular point} ($x_0\in Reg$), and in the second case that $x_0$ is 
a \textit{singular point} ($x_0\in\Sigma$).  The singular part of the free boundary $\Sigma$ can be decomposed further into $d$ strata $\Sigma^k$, according to the 
dimension $k$ of the kernel of $A_{x_0}$,
$$\Sigma = \Sigma^0 \cup \Sigma^1 \cup \dots \cup \Sigma^{d-1}.$$

The uniqueness of the blow-up limit and the rate of convergence of the blow-ups to this limit determine the fine 
properties of $Reg$ and $\Sigma$. In \cite{C1}, it was shown that $Reg$ is locally a $C^{1,\alpha}$ hypersurface. Combined with an earlier work of Kinderlehrer-Nirenberg \cite{KN}, this implies that $Reg$ is analytic. In 
\cite{C2}, the uniqueness of the blow-up limits was established at points in $\Sigma$. This implies that each $\Sigma^k$ is locally covered by a $C^1$ submanifold of dimension $k$.

The regularity of $\Sigma$ was improved successively through quantitative rates of convergence of the blow-ups. When $d=2$, Weiss obtained the $C^{1,\alpha}$ regularity of the manifolds by introducing the  Weiss monotonicity formula \cite{W}. Based on the same formula, Colombo-Spolaor-Velichkov \cite{CSV1} showed through a log-epiperimetric inequality that in higher dimensions the manifolds are of class $C^{1,\log^\eps}$. So far the best result is due to Figalli-Serra \cite{FS}, who employed techniques from the thin obstacle problem. By applying Almgren's monotonicity formula, they improved $C^{1,\log^\eps}$ to $C^{1,\alpha}$ for the manifold covering the top stratum $\Sigma^{d-1}$. They also showed that each stratum $\Sigma^k$ can be further divided into a \textit{`good' part }$\Sigma^k_g$ and an \textit{`anomalous' part} $\Sigma^k_a$, where the former is covered by $C^{1,1}$ manifolds, and the latter is of lower dimension. 

In \cite{SY}, we obtained a rate of convergence for the quadratic blow-ups centered at a point on $\Sigma$, through an iterative scheme which 
is based on the linearization of the problem. 
At a point on the top stratum $\Sigma^{d-1}$, the linearized equation is given by the thin obstacle problem. 
At points on the lower strata $\Sigma^k$ $k\le d-2$, the linearized equation   is a degenerate thin obstacle problem, where the obstacle supported on a subspace of null capacity. 
The results in \cite{SY} give same general regularity of the strata $\Sigma^k$ as in \cite{FS}, 
that is, locally, the top stratum $\Sigma^{d-1}$ is included in a $C^{1,\alpha}$ hypersurface, and, for $k \le d-2$, each lower stratum $\Sigma^k$ is included in a $k$-dimensional $C^{1,\log^\eps}$ submanifold. 

The estimates in \cite{SY} do not rely on monotonicity formulae, 
and are based on the propagation of monotonicity and convexity together with the regularity properties of the linearized problems. 
They apply to the more general \textit{nonlinear obstacle problem}
\begin{equation}\label{eq2}
F(D^2u)=\chi_{\{u>0\}}, \quad \quad u\ge0,
\end{equation} 
with $F\in C^{1,\alpha}$ a convex, uniformly elliptic operator, and \eqref{eq2} is understood in the viscosity sense, see \cite{CC}. 
The convexity of $F$ guarantees that the pure second derivatives $u_{ee}$ are super solutions, while $F \in C^1$ implies that the linearized problems have 
constant coefficients. 
The regularity of $Reg$ for the problem \eqref{eq2} was addressed by  K.A. Lee in \cite{L}, where the results of 
Caffarelli from \cite{C1} were extended.

In this paper, we refine the results in \cite{SY}. Similar to Figalli-Serra \cite{FS}, we obtain a further decomposition of each stratum
$\Sigma^k$ for the nonlinear problem \eqref{eq2} 
as $\Sigma^k=\Sigma^k_g \cup \Sigma^k_a,$ with $\Sigma^k_g$ more regular and $\Sigma^k_a$ of lower dimension.

Our main result is the following:

\begin{thm}\label{T0}
Let $u$ be a solution to \eqref{eq2}. 

Each stratum of the singular part can be decomposed as
$$\Sigma^k=\Sigma^k_g \cap \Sigma^k_a,$$
with $\Sigma_g^k$ locally covered by a $k$-dimensional $C^{1,1-}$ manifold, and $\Sigma_a^k$ relatively open in $\Sigma^k$, and

a) if $k \le d-2$, then $\Sigma^k_a$ has dimension at most $k-1$; in particular $\Sigma_a^0=\emptyset$ and $\Sigma^1_a$ consists of isolated points in $\Sigma^1$ (that can have 
accumulation points on $\cup_{k \ge 1} \Sigma^k$);

b) if $k=d-1$, then $\Sigma^k_a$ has dimension at most $d-3$; in dimension $d=3$, $\Sigma^2_a$ is finite. 
\end{thm}  

By a $C^{1,1-}$ manifold, we understand a manifold that is of class $C^{1,1-\eps}$ for any $\eps>0$. As mentioned above, Figalli and Serra established 
the decomposition  for solutions to \eqref{eq1}. Their result includes the end-point $C^{1,1}$ regularity of the manifolds. 

Theorem \ref{T0} is a consequence of Whitney's extension theorem, Theorems \ref{T1} and \ref{T2}, where explicit point-wise $C^2$ estimates for 
$u$ are derived at various points on $\Sigma$. The set of anomalous points on $\Sigma$ consists of those 
points where 
$u$ fails to be $C^{2,\alpha}$, for some $0<\alpha <1$, see Definitions \ref{AN1} and \ref{AN2}.  

The proof of  part b) in Theorem \ref{T0} is based on related results for the thin obstacle problem. One important ingredient is from a recent work of Focardi-Spadaro \cite{FoS}. For the thin obstacle problem, they showed that  the set of free boundary points with 
frequencies between 2 and 3 is of dimension at most $d-3$. Another key ingredient is the stability of free boundary points of low frequencies. This stability is obtained by establishing a general frequency gap near all integers for the thin obstacle problem, which is interesting in its own: 
\begin{thm}[Frequency gap in the thin obstacle problem]\label{FGIntro} For each $m \in \mathbb N$, there exits a constant $\alpha_m>0$ small, depending only on the dimension $d$ and $m$, so that $ (m-\alpha_m,m+\alpha_m) \setminus\{m\}$ does not contain any admissible frequency for the thin obstacle problem.  
\end{thm}
We only use this result for $m=2$ and $3$. Near even integers, this gap was already known from Colombo-Spolaor-Velichkov \cite{CSV2}. 
Our proof is different and  applies to odd integers as well. 

The paper is organized as follows. In Section 2, we review the results from \cite{SY} and explain how we will quantify them more precisely. In 
Section 3, we prove Theorem \ref{T1},which implies part a) of Theorem \ref{T0}. In Section 4, we review the relevant results for the thin obstacle problem and prove Theorem \ref{FGIntro}. In Section 5, we obtain Theorem \ref{T2} which corresponds to part b) of Theorem \ref{T0}.

\section{Some results of the first paper}

In this section, we state  our hypotheses and review some results in \cite{SY}.

Let $u$ be a solution to the nonlinear obstacle problem \begin{equation}\label{OP}
\begin{cases}F(D^2u)=\chi_{\PosS}, &\\u\ge0,&
\end{cases} \text{ in $B_1$,}
\end{equation} where $F$ is a fully nonlinear elliptic operator, satisfying the assumptions in the follow subsection.

\subsection{Assumptions on $F$ and consequences} Let $\mathcal{S}_d$ denote the space of $d$-by-$d$ symmetric matrices. The operator $F:\mathcal{S}_d\to\R$ satisfies: 
\begin{equation}\label{FirstAssumption}F(0)=0; \text{ }F\text{ is convex};\end{equation} 
\begin{equation}\label{SecondAssumption}F \text{ is $C^{1,\alpha_F}$ for some $\alpha_F\in(0,1)$ with $C^{1,\alpha_F}$ seminorm } [F]_{C^{1,\alpha_F}}\le C_F;  \end{equation}
there is a constant $1\le\Lambda<+\infty$ such that \begin{equation}\label{Ellipticity}\frac{1}{\Lambda}\|P\|\le F(M+P)-F(M)\le\Lambda\|P\|\end{equation} for all $M,P\in\mathcal{S}_d$ and $P\ge 0.$

We call a constant \textit{universal} if it depends only on the dimension $d$, the elliptic constant $\Lambda$ and $C_F$, $\alpha_F$.

For a $C^2$ function $\varphi$, define the \textit{linearized operator} $L_\varphi:\mathcal{S}_d\to\R$ by \begin{equation*}\label{LinearizedOperator}L_\varphi(M)=\sum_{ij} F_{ij}(D^2\varphi)M_{ij},\end{equation*}where $F_{ij}$ denotes the derivative of $F$ in the $(i,j)$-entry, and $D^2\varphi$ is the Hessian of $\varphi$. One consequence of convexity is \begin{equation}\label{CompareWithLinearizedEquation}
L_v(w-v)\le F(D^2w)-F(D^2v)\le L_w(w-v).
\end{equation} 

Let $v$ be a solution to $$F(D^2v)=1 \text{ in $B_1$}.$$ The Evans-Krylov estimate and \eqref{SecondAssumption} imply that 
$$\|v\|_{C^{3,\alpha}(B_{1/2})}\le C \|v\|_{\mathcal{L}^\infty(B_1)}$$
with $C$, $\alpha>0$ universal. The same estimate holds for the difference between $v$ and a quadratic polynomial $p$ with $F(D^2p)=1,$ that is, 
\begin{equation}\label{v-p}
\|v-p\|_{C^{3,\alpha}(B_{1/2})}\le C\|v-p\|_{\mathcal{L}^\infty(B_1)}.
\end{equation}

In particular, if $u$ solves \eqref{OP}, then in $\PosS$ we can differentiate the equation in a unit direction $e\in\Sph$, and then use convexity of $F$ to get\begin{equation}\label{EquationForDerivatives}
L_u(D_eu)=0,\text{ } L_u(D_{ee}u)\le 0 \text{ in $\PosS.$}
\end{equation} 

It is standard to estimate the higher norms of the difference between two solutions of $F$ in terms of the $L^\infty$ norm of their difference (see Proposition 2.2 in \cite{SY}).
\begin{prop}\label{EstimateForDifference}
Let $F$ be as above. 

If $v$ and $w$ solve $$F(D^2v)=F(D^2w)=1 \text{ in $B_1$},$$ then 
$$\|v-w\|_{C^{2,\alpha}(B_{1/2})}\le C\|v-w\|_{\mathcal{L}^\infty(B_1)}$$ 
with $\alpha\in(0,1)$ universal, and $C$ further depending on
$\|v\|_{\mathcal{L}^\infty(B_1)}$ and $\|w\|_{\mathcal{L}^\infty(B_1)}.$

\end{prop}

\subsection{The free boundary}

Lee established in \cite{L} the optimal regularity $u \in C_{loc}^{1,1}$, and analyzed the free boundary $\partial\{u>0\}$ by a blow-up analysis. The results can be summarized as follows.

\begin{thm}\label{C11}
Let $u$ be a solution to \eqref{OP}. Assume that $ 0 \in \partial \{u>0\}$. Then $$\|u\|_{C^{1,1}(B_{1/2})}\le C$$ for some $C$ universal. Moreover, the blow-up rescalings 
$$u_r(x):=r^{-2} u(r x)$$
converge locally uniformly along subsequences of $r_j \to 0$ to global solutions that are either 

1) $$\mbox{a half-quadratic } \quad  \frac 12c_\xi [(x \cdot \xi)^+]^2 \quad \mbox{for some $ \xi \in \Sph$ and $c_\xi\in\R$ with $ F (c_\xi\xi \otimes \xi) =1$, or }$$

2) $$ \mbox{a whole quadratic} \quad \frac 12 x^T A x \quad \mbox{ with $A \ge 0$, $F(A)=1$.} $$

\end{thm} 

If there is a blow-up sequence which ends up in case 1), then we say that $0$ is a {\it regular} free boundary point and write $0 \in Reg$. Otherwise, any blow-up limit is a whole quadratic, and we say that $0$ is a {\it singular} point and write $0 \in \Sigma$. 

\begin{rem}\label{r1}
If $0\in\Sigma$, then the zero set $\{u=0\}$ cannot contain a nontrivial cone with vertex at $0$. In particular, the solution $u$ cannot be nondecreasing in an open cone of directions near the origin. 
\end{rem}

These definitions imply that the free boundary decomposes into the regular part and the singular part $$\partial\PosS=Reg \cup\Sigma.$$

It was shown in \cite{L} that $Reg$ is open in $\partial \{u>0\}$, and that  the blow-up limit at a regular point is unique. Moreover, in a neighbourhood of a point from $Reg$, the free boundary $\partial \{u>0\}$ is a $C^{1,\alpha}$ hypersurface which separates the coincidence  set, $\{u=0\},$ from the positivity set $\{u>0\}$.

In \cite{SY}, we analyzed the behavior of $u$ near points in $\Sigma$, and showed the uniqueness of the blow-up limit profile in 2). 
As a consequence, the singular set $\Sigma$ can be decomposed into $d$ disjoint sets (\textit{strata}), depending on the dimension of the kernel of the blow-up polynomial $\frac 12 (x-x_0)^T A_{x_0} (x-x_0)$ at a point $x_0 \in \Sigma$:
$$\Sigma= \Sigma^0 \cup \Sigma^1 \cup ... \cup \Sigma^{d-1}$$
with
 $$\Sigma^k:=\{x_0 \in \Sigma|\dim ker(A_{x_0})=k\}.$$

\subsection{Results in \cite{SY}} We recall the main results and some notations of \cite{SY}.

First we define the class of polynomial solutions to the obstacle problem, and include also the \textit{ convex} polynomials that do not necessarily satisfy the non-negative constraint.

\begin{defi}\label{QuadraticSolution}The class of \textit{quadratic solutions} is defined as $$\Qua=\left \{p:\quad p(x)=\frac{1}{2}x^TAx, \quad A\ge 0, \quad F(A)=1 \right \}.$$
The larger class of \textit{unconstraint convex quadratic solutions} is defined as $$\UQua=\left \{p: \quad p(x)=\frac{1}{2}x^T Ax+b\cdot x, \quad A\ge 0, \quad F(A)=1 \right\}.$$
\end{defi} 
Note that for a polynomial $p\in\UQua$, its convexity and the ellipticity \eqref{Ellipticity} of $F$ imply $$D^2p\le C \, I,$$ for some universal $C$.

The goal is to keep track of the polynomial approximations for a solution $u$ of \eqref{OP} in dyadic balls.

\begin{defi}\label{ApproxClass}
Given $\eps,r\in(0,1)$ and $p\in\UQua$, we say that $u$ is \textit{$\eps$-approximated by the polynomial $p$ in $B_r$}, and use the notation $$u\in\mathcal{S}(p,\eps,r)$$
  if $$u \text{ solves \eqref{OP} in $B_r$,} \quad \quad |u-p|\le\eps r^2 \text{ in $B_r$},$$ and 
  \begin{equation}\label{convexity}
  D^2u\ge-c_0\eps \,  I\text{ in $B_r,$}
  \end{equation} where $c_0=\frac{1}{16\Lambda^2}$.

\end{defi} 

We now state the main lemmas from \cite{SY}, which provide good rates of the quadratic approximations in dyadic balls. 
They imply the uniqueness of blow-ups together with regularity properties of each strata. 

Depending on the size of the eigenvalues of $D^2p$ with respect to $\eps$, there are two cases to consider.  

For a matrix $M\in\mathcal{S}_d$, we denote its eigenvalues by $$\lambda_1(M)\ge\lambda_2(M)\ge\dots\ge\lambda_d(M).$$
The first case deals with the situation when $u$ is $\eps$-approximated by a polynomial $p \in \UQua$ in $B_1$ 
with $\lambda_2(D^2 p)\gg\eps$. In this case, we expect $0 \in \Sigma^{k}$ with $k \le d-2$.

\begin{lem}[Lemma 5.1 in \cite{SY}, Quadratic approximation near codimension $\ge 2$]\label{QuadraticApproxLowerStratum}
Suppose $$u\in\SPE \quad \mbox{ with $0\in\Sigma$ and $p\in\UQua.$}$$ 
There are universal constants $\kappa_0$ large, $\bar{\eps}$ small, and $\rho\in(0,\frac 12)$ such that if $\eps<\bar{\eps}$  and  
\begin{equation}\label{l2f}
\lambda_2(D^2p)\ge\kappa_0\eps,
\end{equation} 
then $$u\in \mathcal{S}(p',\eps',\rho)$$ for some $p'\in\UQua,$ and one of two alternatives happens for $\eps'$:
\begin{enumerate}
\item{$$\eps'\le(1-c)\eps \quad \mbox{ for a universal $c \in(0,1)$; or}$$}
\item{$$\eps'\le\eps-\eps^\mu \quad \mbox{ and} \quad  (u-h)(\frac{1}{2}\rho e_1)\le C(\eps-\eps'),$$ 
for some universal constants $\mu, C>1,$ where $h$ is the solution  to the unconstrained problem
\begin{equation}\label{hdef}
\begin{cases}
F(D^2h)=1 &\text{ in $B_1$,}\\ h=u &\text{ on $\partial B_1$.}
\end{cases}
\end{equation}
}
\end{enumerate}
\end{lem}

We make a few comments on Lemma \ref{QuadraticApproxLowerStratum}. 

The situation (1) is consistent with a $C^{2,\alpha}$-estimate for $u$ at $0$ for some $\alpha\in(0,1)$ depending on $c$ and $\rho$. The situation (2) gives a much slower improvement for $\eps(r)$ as $r \to 0$, and it 
implies only a $C^{2,\log^c}$-estimate for $u$ at $0$. In Section \ref{kg2}, by quantifying Lemma \ref{QuadraticApproxLowerStratum} more precisely, we give a refined characterization of the strata $\Sigma^k$ with $k \le d-2$.

Firstly, given any $\beta <1$, the estimate in (1) can be improved to $ \eps' \le \rho^\beta \eps$, consistent with the $C^{2,\beta}$ scaling. Secondly, if alternative (2) holds for some ball $B_r$, we show that it continues to hold for all dyadic balls $B_{\rho^{m}r}$. As a result, we show that $u$ cannot be point-wise $C^{2,\alpha}$ at the origin for any $\alpha >0$. In this second case, using the notation as in \cite{FS}, we say that $0$ is an anomalous point of $\Sigma^k$ ($0\in\Sigma^k_a$). It follows that $\Sigma^k_a$ is open in $\Sigma^k$, and has dimension at most $k-1$. The remaining part of $\Sigma^k$, the `good' part $\Sigma^k_g$, can be covered by a $C^{1,1-}$ submanifold of dimension $k$.

The key observation in the proof of Lemma \ref{QuadraticApproxLowerStratum} is that the hypothesis \eqref{l2f} implies that $u$ and $h$ are $o(\eps)$-close away from a codimension 2 subspace. For a positive constant $\eta$, we define the following cylinder $$\mathcal{C}_\eta=\{|(x_1,x_2)|\le\eta\}.$$ 
\begin{lem}[Lemma 5.2 in \cite{SY}]\label{OutsideCylinder}
Let $u, p, h$ be as in Lemma \ref{QuadraticApproxLowerStratum}.

Given $\eta$ small, there is $\kappa_\eta$, depending on universal constants and $\eta$, such that if $\kappa_0 \ge\kappa_\eta$, then $$\|u-h\|_{C^2(B_{1/2}\backslash\mathcal{C}_{\eta})}\le\eta\eps$$ for all $\eps$ small, depending on $\eta.$ 
\end{lem} 

The idea of the proof of Lemma \ref{QuadraticApproxLowerStratum} is to show that $h$ is essentially tangent of order 2 at the origin i.e. $h(0)$, $|\nabla 
h(0)| = o(\eps)$ (otherwise $0$ would be in $Reg(u)$). Then alternatives (2) and (1) are dictated by whether or not $D^2 h(0)$ has negative eigenvalues of size $\eps$.

We remark that the estimate for $u-h$ at the point $\frac \rho 2 e_1$ is used towards the convergence of the corresponding Hessians $D^2 h(0)$ as we zoom in using dyadic balls. Suppose that we are in the slow improvement situation (2). Let $h'$ denote the solution to \eqref{hdef} in the ball $B_\rho$. By the  maximum principle, we have $u \ge h' \ge h$ in $B_\rho$. The Harnack inequality for the difference $h'-h$, and Proposition \ref{EstimateForDifference} imply
\begin{align}\label{h'-h}
|D^2h'(0)-D^2 h(0)| & \le C_\rho \|h'-h\|_{L^\infty(B_\rho)} \nonumber \\
& \le C (h'-h)(\frac{1}{2}\rho e_1) \le C (u-h)(\frac{1}{2}\rho e_1)\le C_1(\eps-\eps') ,
\end{align}
for some  $C_1$, depending on universal constants and $\rho$. 

We now state the main quadratic approximation lemma from \cite{SY} in the case when only $\lambda_1(D^2 p)$ is much larger than $\eps$. 

\begin{lem}[Lemma 4.1 in \cite{SY}. Quadratic approximation near codimension 1] \label{QuadraticApproxTopStratum}
Suppose $$0 \in \Sigma, \quad \quad u\in\SPE \quad \mbox{with $p\in\Qua$}$$ 
and $$\lambda_2(D^2p)\le \kappa\eps$$
for some constant $\kappa>0$. 

There are constants $\bar{\eps},c \in(0,1)$ and $\bar{\rho}\in(0,1/2)$, depending on universal constants and $\kappa$, such that if $\eps<\bar{\eps}$, then $$u\in\mathcal{S}(p',\eps',\rho) \quad \mbox{ for some $p'\in\Qua$,}$$
with $$\eps'=(1-c)\eps, \quad \quad \mbox{and} \quad \rho\in(\bar{\rho},1/2).$$
\end{lem} 

The idea of its proof is to show that the normalized error $\frac 1 \eps (u-p)$ is well approximated by a solution to the linearized problem,  the thin obstacle problem. The frequency at the origin for this thin obstacle problem cannot be 1 or $\frac 32$, since otherwise $0$ is either interior to $\{u=0\}$ or   in  $Reg$. Then the frequency has to be at least 2, which implies the geometric improvement $ \eps \to \eps'$ above.  

In Section \ref{ke1}, we quantify Lemma \ref{QuadraticApproxTopStratum} more precisely and obtain a refined characterization of the stratum $\Sigma^{d-1}$. Depending the frequency at the origin in the thin obstacle problem, we have a dichotomy. 

If the frequency is higher than or equal to $3$, we show that we can replace $ \eps'  = \rho^\beta \eps$ (for any fixed $\beta <1$), consistent with the $C^{2,\beta}$ scaling. Otherwise,  we leave $\eps'$ as above but the rescaled error has `frequency' between 2 and 3. In addition, if the second alternative holds for some ball $B_r$, then  it continues to hold for all other smaller dyadic balls. In this case, using the notation as above, we say that $0$ is an anomalous point ($0\in\Sigma_a^{d-1}$). It follows that $\Sigma^{d-1}_a$ is open in $\Sigma^{d-1}$, and has dimension at most $d-3$. The remaining part of $\Sigma^{d-1}$, the `good' part $\Sigma^{d-1}_g$, can be covered by a $C^{1,1-}$ hypersurface.

Since there is no monotonicity formula for the nonlinear problem, we rely on geometric information of the solution  in terms of its monotonicity and convexity. Roughly,  the following lemmas state that if $u$ is monotone/convex in $B_r$ away from a strip of width $\eta \ll r$, then $u$ is monotone/convex in $B_{r/2}$.

\begin{lem}[Lemma 3.2 in \cite{SY}]\label{Monotonicity1}
Suppose $u\in\mathcal{S}(p,\eps,r)$  satisfies the following for some constants $K$, $\sigma$, and $0<\eta<r$,  and a direction $e\in\Sph$:
$$D_eu\ge -K\eps \text{ in $B_r$}, \quad \quad \mbox{ and} \quad  D_eu\ge\sigma\eps \text{ in $B_r\cap\{|x_1|\ge\eta\}$}.$$

There is $\bar{\eta}$, depending on universal constants, $r$, $\sigma$ and $K$, such that if $\eta\le\bar{\eta}$, then $$D_eu\ge 0 \text{ in $B_{r/2}$.}$$
\end{lem} 

\begin{lem}[Lemma 3.4 in \cite{SY}]\label{Convexity}
Suppose $u\in\SPE.$ There is a universal constant $C$ such that if $\Dee p\ge C\eps$ along some direction $e\in\Sph$, then $$\Dee u\ge 0 \text{ in $B_{1/2}$.}$$
\end{lem}

\begin{rem}\label{Monotonicity2}
The hypothesis $D_e u \ge \sigma \eps$ in $B_r\cap\{|x_1|\ge\eta\}$ in Lemma \ref{Monotonicity1} can be relaxed to 
$$D_e u\ge 0 \quad \mbox{in $B_r\cap\{|x_1|\ge\eta\}$, and} \quad D_e u \ge \sigma \eps \quad \mbox{in $B_r\cap \{u \ge c r^2\}$,}$$
for some $c$ small, universal (see Lemma 3.3 in \cite{SY}). 
\end{rem}

\section{A refinement of Lemma \ref{QuadraticApproxLowerStratum}}\label{kg2}

In this section, we refine Lemma \ref{QuadraticApproxLowerStratum}. The proof follows the lines of the one in \cite{SY}, with a few modifications towards the end. For convenience of the reader, we reproduce the whole argument, leaving out only  technical points that are identical with the ones in \cite{SY}.  

Since the solution $u$ is $\eps$-approximated by $p$ in $B_1$ with $\lambda_2(D^2p) \gg \eps$, the coincidence set $\{u=0\}$ concentrates around a subspace of codimension at least 2. As a result, the coincidence set has small capacity, and we expect to approximate $u$ by the solution to the following unconstrained problem.  

For $0<r<1$, let $h_r$ be the solution to the following:
\begin{equation}\label{hrdef}
\begin{cases}
F(D^2h_r)=1 &\text{ in $B_r$,}\\ h_r=u &\text{ on $\partial B_r$,}
\end{cases}
\end{equation}
and denote its Hessian at $0$ by
\begin{equation}\label{Ar}
A_r= D^2h_r(0).
\end{equation}

With these notation, Lemma  \ref{QuadraticApproxLowerStratum} can be refined to

\begin{lem} \label{L1}
Suppose $$u\in\SPE \quad \mbox{ with $0\in\Sigma$ and $p\in\UQua.$}$$
Given any $\beta\in(0,1)$,
there are constants $\kappa_0$, $C$ large, $\bar{\eps}$, $c_1$, $\rho$ small, depending on $\beta$ and the universal constants, such that if $\eps<\bar{\eps}$  and  
\begin{equation}\label{l2}
\lambda_2(D^2p)\ge\kappa_0\eps,
\end{equation} 
then $$u\in \mathcal{S}(p',\eps',\rho) $$
for some $p'\in\UQua$ with $$|D^2p'-A_1|\le C\eps,$$ and 
\begin{enumerate}
\item{ if $|A_1^-| \le c_1 \eps $  then $$\eps'=\eps \rho^\beta.$$}
\item {if $|A_1^-| \ge c_1 \eps $ then 
$$\eps' \le\eps-\eps^\mu \quad \mbox{ and} \quad |A_\rho^-| \ge c_1 \eps'$$ 
for some universal constant $\mu>1$.}
\end{enumerate}

Moreover, in both cases, we have
$$|A_1 - A_\rho| \le C(\eps-\eps').$$
\end{lem}

\begin{rem}\label{IndefinitelyApplicable}
 Beginning with an initial approximating $p_1$ with error $\eps_1$, Lemma \ref{L1} is applied to dyadic balls of radii $r=\rho^m$, and gives approximation of $u$ by a sequence of parabolas $p_m$ with decreasing errors $\eps_m$. By the estimates on $|D^2p'-A_1|$ and $|A_1-A_\rho|$, the difference between $D^2p_m$ and $D^2p$ is at most C$\eps_1$. Consequently, if we begin with $D^2p_1\ge\kappa\eps_1$ for some $\kappa\ge\kappa_0+C$, then condition \eqref{l2} holds for all $p_m$, and Lemma \ref{L1} can be applied inductively. 
\end{rem}

Lemma \ref{L1} is more precise than Lemma \ref{QuadraticApproxLowerStratum}. 

In case (1), the improvement of $\eps'$ is consistent with the $C^{2,\beta}$ scaling for $\beta$ arbitrarily close to 1. 
In case (2), the lower bound on $|A_{\rho}^-|$ shows that if this alternative happens at one scale $r$, then it happens for all finer scales $\rho^mr.$ 

If  alternative (2) happens, then  the solution $u$ cannot be $C^{2,\alpha}$ at the origin for any $\alpha>0$, and we say that $0$ is an anomalous point of $\Sigma$.

\begin{defi}\label{AN1}
We say that $x_0 \in \Sigma$ is \textit{anomalous} and write $$x_0 \in \Sigma_a$$ if 
$$ \|u - \frac 12 (x-x_0)^T D^2 u(x_0) (x-x_0)\|_{L^\infty(B_r(x_0))} \ge r ^{2+\alpha} $$
for any $\alpha>0$ and all $r$ small. 

We denote the complement of the anomalous part, the good part, by 
$$\Sigma_g:=\Sigma \setminus \Sigma_a.$$

For each $ k \le d-2$, we denote $$\Sigma^k_a:= \Sigma^k \cap \Sigma_a, \quad \quad \Sigma^k_g:=\Sigma^k \cap \Sigma_g.$$
\end{defi}  

Alternative (1) in Lemma \ref{L1} implies that the solution u is $C^{2,\beta}$ at all points in $\Sigma^k_g$ with $k\le d-2$. This gives the desired $C^{1,1-}$ covering of the good part. 

On the other hand, if $0\in\Sigma^k_a$, we have the following lemma, which states that   $\Sigma$ coincides with $\Sigma_a$ near $0$, and that $\Sigma^{k}_a$ concentrates near a $(k-1)$-dimensional space.
\begin{lem}\label{L11}
Under the assumptions in Lemma \ref{L1}, and assume that alternative (2) holds. Then $$\Sigma \cap B_\rho \subset \Sigma_a$$ and, after rescaling, alternative (2) holds in any ball $B_r(x_0)$ with $x_0 \in \Sigma \cap B_\rho$ and $r \le \rho$. 

Moreover, if 
\begin{equation}\label{l22}
\lambda_{d-k}(D^2p) \ge \kappa \eps \quad \quad \mbox{for some $k \le d-2$},
\end{equation}
then $\Sigma \cap B_\rho$ is in a $\sigma (\kappa)$-neighborhood of a subspace of dimension $k-1$, with $\sigma(\kappa) \to 0$ as $\kappa \to \infty$.
\end{lem}

Suppose that $0\in\Sigma^k$ for some $k\le d-2$, then the unique blow-up limit, say, $p_1$, satisfies $$\lambda_1(D^2p_1)\ge\lambda_2(D^2p_1)\ge\dots\ge\lambda_{d-k}(D^2p_1)>0.$$Consequently, after an initial rescaling, we can assume $u\in\mathcal{S}(p_1,\eps_1,1)$ for a small $\eps_1$ satisfying $\lambda_{d-k}(D^2p_1)\ge 2\kappa_0\eps_1.$ By the same reasoning as in Remark \ref{IndefinitelyApplicable}, $\lambda_{d-k}(D^2p_m)\ge 2\kappa_0\eps_1$ for all subsequent approximating parabolas $p_m$. On the other hand, the approximating errors $\eps_m\to 0$. This implies $$\lambda_{d-k}(D^2p_m)/\eps_m\to\infty.$$ If alternative (2) as in Lemma \ref{L1} ever happens, then Lemma \ref{l22} implies that at smaller and smaller scales, $\Sigma$ is trapped in a neighborhood of a $(k-1)$-dimensional subspace with vanishing width. 
Standard covering arguments imply that $\Sigma^k_a$ has Haussdorff dimension $k-1$. In particular for $k=0$ we have $\Sigma^0_a=\emptyset$, and for $k=1$, $\Sigma^1_a$ consists of isolated points inside $\Sigma^1$.

If we start with the situation of Lemma \ref{L1} with, say $\beta=1/2$, then after a few iterations of the lemma we may apply the estimates of the Lemma \ref{L1} for a different value of $\beta$ much closer to 1, and so on.

We summarize these  in the following theorem, which gives part a) of our main result Theorem \ref{T0}.

\begin{thm}\label{T1}
Assume that $\eps \le \bar \eps_0$, 
$$u\in\SPE \quad \mbox{and} \quad \lambda_2(D^2p)\ge \bar{\kappa}_0 \eps,$$
with $\bar{\eps}_0$ small, $\bar{\kappa}_0$ large universal constants. Then in $B_{1/4}$ we have

a) $\Sigma = \cup_{k \le d-2} \Sigma^k$ and $u$ is $C^{2,\log ^c}$ on $\Sigma$: 

for each $x_0 \in \Sigma$, $\exists \, q_{x_0}$, a quadratic polynomial, with $|D^2 q_{x_0} - D^2 p| \le C \eps$ and
$$|u-q_{x_0}|(x)\le C|x-x_0|^2|\log|x-x_0||^{-c}.$$

b) $u$ is uniform $C^{2,\beta}$ for any $\beta<1$ on the non-anomalous set $\Sigma_g$: 
$$|u-q_{x_0}|(x)\le C(\beta) \eps |x-x_0|^{2+\beta} \quad \quad \mbox{if} \quad x_0 \in \Sigma_g.$$

c) The anomalous set $\Sigma_a^k$ is open in $\Sigma^k$ and has Haussdorff dimension $k-1$. In particular, $\Sigma^0_a=\emptyset$, and $\Sigma_a^1$ is discrete locally (can have accumulation points on $\Sigma$).

\end{thm}

Theorem \ref{T1} follows directly from Lemmas \ref{L1} and 
\ref{L11} and the discussion above. The rest of the section is devoted towards the proof of Lemma \ref{L1}. By examining this proof, we deduce Lemma \ref{L11}. 

Let us assume $\lambda_k(D^2p) \gg \eps \text{ for some $k\ge 2$,}$ and that $D^2p$ has ordered eigenvalues on the diagonal. Then the coincidence set  $\{u=0\}$ concentrates around the set
\begin{equation}\label{n-k}
\{x'=0 \in \R^k\}, \quad \mbox{where} \quad x':=(x_1,..,x_k), \quad x'':=(x_{k+1},...,x_d).
\end{equation} 

For simplicity of notation we denote $h_1$ by $h$.

The normalized error $$\hat u=\frac 1 \eps (u-p)$$ solves an obstacle problem with the obstacle $$\hat O=-\frac 1 \eps p.$$ Since the capacity of the coincidence set converges to 
$0$ as $\eps \to 0$, we expect to approximate $\hu$ by  $$\hat h := \frac 1 \eps (h-p).$$

By Lemma \ref{OutsideCylinder}, this is true outside a small tubular neighborhood $C_\eta$ around the $x''$-subspace.
Inside $C_\eta$, however,  the difference between $\hat h$ and $\hat u$ could be of order 1, and $\hat u$ might not have a uniform modulus of continuity as $\eps \to 0$. 

Heuristically, as $\eps \to 0$, we end up with limiting functions $\bar u$, $\bar O$ and $\bar h$, that satisfy that $|\bar h|$, $|\bar u|$ and $\max \bar O$ are all bounded by 1 in $B_1$, and

\noindent 1) $\bar h$ is a solution to a constant coefficient elliptic equation, 

\noindent 2) the obstacle $\bar O$ is a concave quadratic polynomial supported on the $x''$-subspace, extended to $-\infty$ outside its support, 

\noindent 3) $\bar u=\max\{\bar h, \bar O \}$, which can be discontinuous.

The improved quadratic error for $\hat u$ follows from the $C^{3}$ estimate of $\bar h$ at the origin. However, first we need to establish that $0 \in \Sigma$ essentially implies that $\bar h$ and $\bar O$ are tangent of order 1 at the origin in the $x''$ direction. The dichotomy in Lemma \ref{L1} depends on whether or not $\bar O$ separates (quadratically) on top of $\bar h$ in this direction.

We now give the proof of the main result in this section:
\begin{proof}[Proof of Lemma \ref{L1}]

As discussed above, we define the normalizations $$\hat{u}=\frac{1}{\eps}(u-p),  \text{ }\hat{h}=\frac{1}{\eps}(h-p)  \text{ and }\hat{O}=\frac{1}{\eps}(0-p).$$
Then in $B_1$, we have $$-1\le\hat{h}\le\hat{u}\le 1, \quad \quad \hat u(0)=\hat O(0)=0,$$
and \eqref{v-p} implies \begin{equation}\label{hC2}\|\hat{h}\|_{C^{3,\alpha}(B_{3/4})}\le C\end{equation} for some universal constant $C$. 
By Lemma \ref{OutsideCylinder}, for any small parameter $\eta>0$ we have
\begin{equation}\label{OutsideCylinder2}
\|\hat u- \hat h\|_{C^2(B_{1/2}\backslash\mathcal{C}_{\eta})}\le\eta, \quad \quad \quad \mathcal{C}_\eta=\{|(x_1,x_2)|\le\eta\}.
\end{equation}
provided that $\kappa_0 \ge \kappa_\eta$ and $\bar \eps \le \eps_\eta$, with $\kappa_\eta$, $\eps_\eta$ depending on universal constants and $\eta$.

We follow the proof of Lemma 5.1 in \cite{SY}, which consists of 6 Steps. The differences appear only in Steps 5 and 6,  where we choose various parameters depending on  $\beta$. For the convenience of the reader, we provide the full argument, with some parts in Steps 1-3 being only sketched.

Before we proceed, we give the outline of the 6 steps. 

We decompose the space $x=(x',x'')$ according to  the curvatures of the obstacle $\hat O$. The curvatures are very negative along the directions in the $x'$-subspace, and are uniformly bounded in the $x''$-subspace. In Steps 1-2 we show that $\hat h$ and $\hat O$ are tangent in the $x''$ direction at the origin up to an arbitrarily small error $\delta$, and deduce that $\hat O$ can only separate quadratically on top of $\hat h$ near the origin. In Step 3, we use barriers to show that the same is true for $\hat u$. In Step 4, we use the $C^{3,\alpha}$ estimate for $\hat h$ to approximate $u$ quadratically in $B_\rho$ by a polynomial $p' \in \UQua$. The lower bound for $D^2 u$ in $B_\rho$, $D^2 u \ge - c_0 \eps' I$ (see \eqref{convexity} in Definition \ref{ApproxClass}) and the choice of $\eps'$ are given in Steps 5 and 6, according to whether the obstacle $\hat O$ separates quadratically on top of $\hat h$ along some direction in the $x''$ subspace. This leads to our dichotomy.

Throughout this proof, there are several parameters $\delta$, $\eta$, $\rho$ to be fixed in the end, depending on $\beta$ and universal constants.
First we will specify the radius $\rho\in(0,1/2)$, and then the parameter $\delta>0$ which can be made arbitrarily small. The parameter $\eta$ from Lemma \ref{OutsideCylinder} that allows us to make $\hat{u}$ and $\hat{h}$ very close to each other, will be chosen to depend on $\delta$. This $\eta$ imposes the choice of $\kappa_0=\kappa_\eta$ as in Lemma \ref{OutsideCylinder}. The parameter $\bar{\eps}$ is chosen after all these. 

Up to a rotation, $p$ takes the form $$p(x)=\frac{1}{2}\sum a_jx_j^2+\sum b_jx_j$$ with $$a_1\ge a_2\ge\dots\ge a_d\ge 0, \quad a_2\ge\kappa_0\eps,$$ and $F(D^2p)=1.$


We introduce some notations. For $\delta$ small to be chosen,  let $k\in\{1,2,\dots, d\}$ be such that 
\begin{equation}\label{a_k}
a_k\ge 2\delta^{-4}\eps>a_{k+1}.
\end{equation} 
Then we decompose the entire space $\R^d$ as $x=(x',x'')$, where $$x'=(x_1,x_2,\dots,x_{k})\text{ and }x''=(x_{k+1},x_{k+2},\dots, x_d).$$

The obstacle $\hat O$ is changing rapidly in the $x'$ direction, and we denote by $\minx$ the point in this direction where its maximum is achieved. This is the same as the minimum point for $p$ in the $x'$ direction. 

Precisely, let $\underline{x}'$ be the minimum point of $x'\mapsto p(x',0)$. Then by \eqref{a_k}, we have  
 $$|\underline{x}'|\le\delta^2, \quad \quad \mbox{and} \quad -\eps\le p(\underline{x}',0)\le0.$$
 
 We write $p$ as the sum of two quadratic polynomials in the $x'$ and $x''$ variables   
 \begin{equation}\label{pdec}
 p(x',x'')=p_1(x'-\minx) - p_1(\minx)+p(0,x''),
 \end{equation} 
 where $p_1 \ge 0 $ is a $2$-homogeneous polynomial
 $$p_1(x')=\frac{1}{2}\sum_{j\le k}a_jx_j^2.$$ 
The obstacle $\hat O$ satisfies
\begin{equation}\label{Oh}
|\nabla_{x''}\hat O|, |D_{x''}^2 \hat O| \, \le C_\delta, \quad \hat O ((\minx,0)) \ge 0.
\end{equation}

 \
 
\textit{Step 1: If $\eta$ and $\eps$ are small depending on $\delta$, then }
\begin{equation}\label{step1}
|\nabla_{x''}(\hat{h}-\hat{O})(0)|<\delta.
\end{equation}

The conclusion can be rewritten as $|\nabla_{x''}h(0)| \le \delta \eps$, and it implies $|b_i| \le C$ if $i >k$.
The idea is to show that otherwise $u$ is monotone in a cone of directions near the $x''$ subspace, and we contradict $0 \in \Sigma$. We sketch the argument.

Suppose  there is $i>k$ such that $D_i(\hh-\hat{O})(0)>\delta.$ 
This estimate can be extended to $D_i(\hh-\hat{O})\ge \frac{1}{2}\delta$ in $B_{r}(0)$ for some $r>0$ depending only on $\delta$. 
By \eqref{OutsideCylinder2}, $D_i(\hu-\hat{O}) \ge \frac{1}{4}\delta$ or equivalently,  $D_i u \ge \frac 1 4 \delta \eps$, in $B_{r}(0)$ outside a strip of width $\eta$. Inside this strip $|D_i u| \le 2 C_\delta \eps$ which is a consequence of $D^2 u \ge -c_0 \eps I$ and \eqref{Oh}.
Now Lemma \ref{Monotonicity1} gives $D_iu\ge 0 \text{ in $B_{r/2}$}$, and by continuity $D_e u \ge 0$ for all unit directions $e$ close to $e_i$.
Thus $\{u=0\}$ contains a cone of positive opening with vertex at $0$, which means $0 \in Reg$, a contradiction.

\

\textit{Step 2: If $\eta$ and $\eps$ are small depending on $\delta$, then} 
\begin{equation}\label{step2}
|\hat{h}(\minx,0)-\hat{O}(\minx,0)|<\delta.
\end{equation}

Note that $\hh(0)\le\hu(0)=0$ and $\hat{O}(\minx,0)\ge 0,$ and then the upper bound for $\hat h - \hat O$ at $(\minx,0)$ follows from $|\minx|\le\delta^2$ and \eqref{hC2}. 

In order to establish the lower bound we prove that if $\hat{h}(\minx,0)-\hat{O}(\minx,0)<-\delta$, then $0 \in Reg$. We sketch the argument.

Using \eqref{OutsideCylinder2} together with the fact that $\hat O$ decays fast in the $x'$ direction while has controlled growth in the $x''$ direction (see \eqref{pdec}), one can show by constructing an explicit upper barrier that $\hat u = \hat O$ (or equivalently $u=0$) at $(\minx, 0)$. By continuity this can be extended to 
$B_{r}(\minx,0) \subset \{u=0\}$ for a small $r>0,$ possibly depending on $\eps$. 

Since $a_j>2\delta^{-4}\eps$ for $j\le k$, we can apply Lemma \ref{Convexity} to get $\Dee u\ge 0$ in $B_{1/2}$ for all unit directions $e$ in a small open cone around the subspace $\{(x',x'')|x''=0\}.$ Using that $u(0)=0$ we conclude that $\{u=0\}$ contains a cone with positive opening and vertex at $0$, hence $0 \in Reg$.

\begin{rem}\label{r22} The argument applies also for a point $(\minx, y'')$ with $|y''|\le 1/2$. If $$\hat{h}(\minx, y'')-\hat{O}(\minx,y'')<-\delta,$$
 then $$\Sigma \cap \{(x',x'')| \quad x''=y''\}\cap B_{1/2} = \emptyset.$$
\end{rem}

\

\textit{Step 3:  If $\eps,\eta$ small depending on $\delta$,  we have 
\begin{equation}\label{step3}\hu\le \hh+ \frac a 2 |x''|^2+ C |x''|^3 + 4\delta \text{ in $B_{1/4}$}
\end{equation}
with $C$ universal, and 
\begin{equation}\label{adef}
a:=\frac 1 \eps |A^-| \le 2 c_0.
\end{equation}}

Recall that $A=D^2h(0)$, and that $c_0$ is the universal constant from Definition \ref{ApproxClass}.

The inequality holds outside $\mathcal C_\eta$ by \eqref{OutsideCylinder2}. It remains to establish it in $\mathcal C_\eta$.

First we use Steps 1 and 2 to show that a similar inequality holds for $\hat O$:
\begin{equation}\label{2041}
\hat O \le \hh+ \frac a 2 |x''|^2+ C |x''|^3 + 3\delta \quad \quad \mbox{in} \quad B_{1/2}.
\end{equation}
Then, as in Step 2, one can use the fast decay of $\hat{O}$ in the $x'$ direction (away from the $(\minx, x'')$ axis) and construct explicit barriers to 
extended the inequality from $\hat O$ to $\hat u$. 
By the same reason, it suffices to prove the inequality \eqref{2041} only on the $(\minx, x'')$ axis with $3 \delta$ replaced by $2 \delta$. This is a consequence of Taylor's expansion in the $x''$ direction from $(\minx,0)$. Indeed, we use Step 2 combined with the estimates
$$|\nabla_{x''}(\hh-\hO)(0)|\le  \delta, \quad D^2(\hh - \hO)(0) \ge - a I,$$ 
that we extend at $(\minx, 0)$ with an extra error of $C|\minx| \le C \delta^2$. This is because $D^3 \hat O=0$, $D^2_{x',x''} \hat O=0$, hence
\begin{equation}\label{hC3}
|D^3 (\hh - \hat{O})| \le C, \quad \quad |D^2_{x'x''}(\hat{h}-\hat O)|\le C.
\end{equation}
Finally, we remark that $a \le 2 c_0$ is a consequence of $D^2 u \ge - c_0 I$. Indeed, in $B_{1/2}\backslash\mathcal{C}_\eta$, \eqref{OutsideCylinder2} gives
$$D^2\hh-D^2\hat{O}\ge D^2\hat{u}-D^2\hat{O}-\eta \, I=\frac{1}{\eps}D^2u-\eta \, I \ge-(c_0+\eta) \, I.$$ By choosing $\eta$ small, we can extend the estimate to the full ball 
\begin{equation}\label{51}D^2(\hh-\hat{O})\ge-2c_0 \, I \text{ in $B_{1/2}$.}\end{equation}

\

\textit{Step 4: 
\begin{equation}\label{step4}
\exists \,  p'\in\UQua \quad \mbox{such that} \quad |u-p'| \le \left( \frac{a}{8c_0} + C_0 \rho \right) \eps  \, \rho^2 \quad \mbox{in} \quad B_\rho,
\end{equation}
if $\delta$ is sufficiently small, depending on universal constants and $\rho$. 
}

\

Define $$q(x)=\frac{1}{2}x\cdot D^2 \hat h(0)x+\nabla \hat h(0)\cdot x.$$
Then \eqref{hC2} implies $$|\hh-\hh(0)-q|\le C\rho^{3} \text{ in $B_{\rho}$.}$$

Using $\hh\le\hu\le\hh+\frac{a}{2}|x''|^2+C |x''|^3+4\delta$ in $B_{1/4}$ from \textit{Step 3}, and $\hu(0)=0$, we find 
$$|\hu-q|\le a \rho^2+2 C \rho^3 + 8\delta \text{ in $B_\rho$.}$$ Pick $\delta$ small such that $\delta \le \rho^3$, then we have $$|u-p-q\eps|\le  (a+ C_0 \rho) \,  \eps \rho^2 \text{ in $B_\rho$,}$$
for some $C_0$ universal.

Define $\tilde{p}=p+q\eps$, then $D^2\tilde{p}=D^2p+\eps D^2q=D^2h(0)=A$. Thus $F(A)=1.$

Next we perturb slightly $\tilde p$ into a convex polynomial $p' \in \UQua$. 

We know $A\ge - a \eps  I\, $, by the definition \eqref{adef} of $a$. Then \eqref{Ellipticity} gives $$1\le F(A^+)\le 1+ \Lambda a \eps, \quad \mbox{and} \quad \frac 1 \Lambda \le |A^+| \le 2 \Lambda.$$ 

Consequently, we can pick $t\in [0, a \Lambda^2 |A^+|^{-1} ]$ such that $$F((1-t \eps)A^+)=1.$$
Denote the new quadratic polynomial $$p'(x):=(1-t \eps)\, \frac{1}{2}x\cdot A^+x+\nabla h(0)\cdot x.$$ 
Then clearly $p'\in\UQua,$ and 
$$|p'-\tilde{p}|\le (t \eps |A^+| + |A^-|) \frac 12 \rho^2 \le a \Lambda^2 \eps\rho^2 \quad \mbox{in} \quad B_\rho.$$ 
Then $$|u-p'|\le (2 a \Lambda^2 + C \rho) \eps\rho^2 \quad \mbox{in} \quad B_\rho,$$ 
and \eqref{step4} is established by recalling the definition of $c_0$ in Definition \ref{ApproxClass}.

\

\textit{Step 5: If $a \le \frac{1}{2}c_0 \rho^\beta \, =:c_1(\beta),$ then (1) holds: $$u \in \mathcal S (p',\eps',\rho)\quad \mbox{ with} \quad  \eps'=\eps \rho^\beta.$$
Here $\rho$ is chosen sufficiently small (depending on $\beta$) so that $C_0 \rho \le \frac 1 4 \rho^\beta$.}

\

The estimate $|u-p'|\le \eps\rho^2$ follows from Step 4. We now show that $$D^2u\ge- c_0\eps \rho^\beta \, I \text{ in $B_\rho.$} $$  The inequality $D^2(\hat h - \hat O) \ge - a I$ at the origin together with $|D^3(\hat h - \hat O)| \le C$ implies that
$$D^2(\hat h - \hat O) \ge - (a + C \rho) I \quad \quad \mbox{in $B_\rho$.} $$
By \eqref{OutsideCylinder2}, a similar inequality holds for $\hat u$ outside the strip of width $\eta$. 
Thus $$w:=\frac 1 \eps u_{ee}^-$$ satisfies
$$ L_u (w) \ge 0, \quad \quad w \le c_0  \quad \mbox{in $B_\rho$,}$$
\begin{equation}\label{HowNegative}w \le a+C \rho + \eta \le \frac 34 c_0 \rho^\beta \quad \mbox{in} \quad B_\rho \setminus \mathcal C_\eta.\end{equation}
Weak Harnack inequality implies that $w \le c_0 \rho^\beta$ provided that the width $\eta$ of the strip is sufficiently small, depending on $\rho$ and the universal constants. This gives the desired lower bound on $D^2u$. 

Recall that $h_\rho$ is the solution to \eqref{hrdef} in $B_\rho$ and $A_\rho=D^2h_\rho(0)$. To complete the proof of alternative a) as in Lemma \ref{L1}, we use  \eqref{h'-h} to get 
\begin{equation}\label{ar1}
|A_\rho - A| \le C(\rho) \, ( u- h)(\frac 12 \rho e_1) \le C \eta \, \eps.
\end{equation}The right-hand side is bounded by $\eps-\eps'$ if $\eta$ is small.

\

\textit{Step 6: If $a > \frac 12 c_0 \rho^\beta$ then the conclusion (2) holds.}

\

We distinguish 2 cases, depending on whether  $a$ is greater than $c_0/4$.

If $a \le c_0/4$, then we can argue precisely as in Step 5, with $\rho^\beta$ replaced by $2a/c_0$, and from Steps 4 and 5 we obtain $$u \in \mathcal S(p',\eps', \rho) \quad \mbox{ with} \quad  \eps'=\frac{2a}{c_0} \eps \le \frac 12 \eps.$$ Moreover by \eqref{ar1}, 
\begin{equation}\label{Ar22}
|A_\rho^-| \ge (a - C \eta) \eps \ge \frac{c_0}{8} \eps' \ge c_1(\beta) \eps'.
\end{equation}

Now we consider the case $a \ge c_0/4$, and get  $\eps' \in [\eps/2, \eps)$. With $a\ge c_0/4$,  \eqref{Ar22} remains valid for any $\eps' \le \eps$. The choice of $\rho$ in Step 5 and $a \le 2 c_0$ imply in \eqref{step4} that $|u-p'| \le \frac 12 \eps \rho^2 \le \eps' \rho^2$ in $B_\rho$. It remains to show the improvement of convexity$$D^2u\ge- c_0\eps'  I \text{ in $B_\rho$} $$ with $\eps'\le \eps-\eps^\mu.$

We show that the improvement $\eps - \eps'$ is at least proportional to $(u-h)(\frac \rho 2 e_1)$. The key observation is that $u-h$ is a subsolution and $D_{ee}u + c_0 \eps$ is a supersolution for the same linearized operator $L_u$, and that the two functions can be compared in the domain $B_1 \cap \{u>0\}$. 

Since $D_{\xi\xi} (\hh -\hO)(0) \le - \frac 12 c_0$ for some unit direction $\xi$, 
then by \eqref{hC3} we conclude $D_{\xi\xi}(\hh-\hO)<-\frac{1}{16}c_0$ in $B_{c}$ for a universal $c>0$. 
Together with \textit{Step 2}, this implies the existence of some $x^*\in B_{1/4}$ such that $(\hh-\hO)(x^*)<-c$ for some universal $c$, that is, $h(x^*)<-c\eps.$
With the universal Lipschitz regularity of $h$, we get $$h<-c\eps \quad \Longrightarrow \quad u-h\ge c\eps \quad \text{ in $B_{c'\eps}(x^*)$}$$ for some small universal $c,c'>0.$ 

Note that $L_h(u-h)\le 0$ in $B_1$ as in \eqref{CompareWithLinearizedEquation}, $u=h$ on $\partial B_1$. We compare $u-h$ to the corresponding solution of the maximal Pucci operator in $B_1 \setminus B_{c'\eps}(x^*)$ and obtain as a consequence of Harnack inequality 
$$u-h\ge \eps^\mu \text{ in $B_{1/2}$,}$$
for some universal $\mu>1$.
Moreover, since $u-h$ solves a linear equation away from $\mathcal C_\eta$, the same argument combined with Harnack inequality imply that
\begin{equation}\label{u-h2}
u-h\ge c (u-h)(\frac 12 \rho e_1) \ge c \eps^\mu \text{ in $B_{1/2}$.}
\end{equation} 
For $e\in\Sph$, we define $$w=\Dee u+c_0\eps.$$ This is a nonnegative function satisfying $L_u(w)\le 0$ in $B_1\cap\PosS.$ Note that $w\ge c_0\eps$ along $\partial\PosS$, and $2 \eps \ge u-h$ in $B_1$, hence
$$w\ge \frac{c_0}{2}(u-h) \quad \text{along $\partial(B_1\cap\PosS)$.}$$
Since  $L_u(w)\le0\le L_u(u-h)$ in $B_1\cap\PosS$, we have $$w\ge\frac{c_0}{2}(u-h) \text{ in $B_1\cap\PosS.$}$$
Combining this with \eqref{u-h2} we find
$$w\ge c(u-h)(\frac{1}{2}\rho e_1)\quad \text{in $B_{1/2}$,}$$ 
which means 
$$\Dee u\ge -c_0\eps+ c(u-h)(\frac{1}{2}\rho e_1) \quad \text{in $B_{1/2}$.}$$
Define the right-hand side to be $-c_0\eps'$, then 
$$\eps':=\eps-\frac{c}{c_0}(u-h)(\frac{1}{2}\rho e_1), \quad \quad  \eps^{2 \mu} \le  \eps - \eps' \le C \eta \eps \le \frac{\eps}{2}.$$ Also, $(u-h)(\frac{1}{2}\rho e_1)=C(\eps-\eps')$ and by \eqref{ar1}, \eqref{Ar22} and Step 4, the second alternative holds.

\end{proof} 

We make a few remarks about Step 6 which are helpful towards Lemma \ref{L11}.

\begin{rem}\label{r61}
We actually proved a stronger bound for $|A_\rho^-|$ than the one stated in alternative (2), which is independent of $\beta$, i.e. $|A^-_\rho| \ge \frac{c_0}{8} \eps'$, see \eqref{Ar22}. 
Thus, if we are in alternative (2), after rescaling back from $B_\rho$ to $B_1$ we end up in the situation $a \ge \frac{c_0}{8}$ of Step 6. 
Then, in either case $a \in [c_0/8, c_0/4]$ or $a \ge c_0/4$, we showed that $ \eps'  \ge \frac 14 \eps$.
\end{rem}

\begin{rem}\label{r62}
After relabeling the constants $\kappa_0$ and $\bar \eps$ to guarantee that the hypothesis \eqref{l2} keeps being satisfied as we apply Lemma \ref{L1} inductively, we obtain 
\begin{equation}\label{pmrm}
u \in \mathcal S(p_m, \eps_m, \rho^m), \quad \quad m \ge 1.
\end{equation} 
If we end up in alternative (2) for some $m_0$, then we remain in alternative (2) for all $m \ge m_0$, and by Remark \ref{r61}, $\eps_{m+1} \ge \frac 14 \eps_m$ if $m \ge m_0 +1 $.
\end{rem}

\begin{rem}\label{r63}
We can relabel the constants $\kappa_0$ and $\bar \eps$ so that in case (2) we also have
$$ u \in \mathcal S(p',\eps',r), \quad \quad |A_r^-| \ge \frac{c_0}{8} \eps', \quad \forall  r \in [\rho^2,\rho].$$
Indeed, we may choose the parameters $\delta$ and $\eta$ sufficiently small so that Step 4, and \eqref{ar1}, \eqref{Ar22} remain valid after replacing $\rho$ by $r$ as above. Thus, by Remark \ref{r62}, if we are in case (2) then $u \in \mathcal S(p_r,\eps_r,r)$ and $|A_r^-| \ge \frac{c_0}{8} \eps_r$ for all $ r \le \rho$. 
\end{rem}

\begin{rem}\label{64}
If we are in case (2), $|A^-| \ge c_1(\beta)\eps$, then the conclusion applies at all points $x_0 \in \Sigma \cap B_\rho$ and not only the origin, since as in $\eqref{ar1}$, $\eqref{Ar22}$ we can deduce that $|A_\rho^-(x_0)| \ge \frac{c_0}{8} \eps'$.
\end{rem}

\

\begin{proof}[Proof of Lemma \ref{L11}]
In view of the Remarks above, it suffices to assume we are in the situation of Step 6 in the proof of Lemma \ref{L1} and establish the following two claims. 

\

{\it Claim 1: $ 0 \in \Sigma_a$.}

First we observe that $u$ cannot be approximated in $L^\infty(B_1)$ by any quadratic polynomial with an error smaller than $c(\beta) \eps$.

Indeed, $|A^-| \ge a \eps$ with $a \ge c_1(\beta) =\frac 12 c_0 \rho^\beta $, 
and $|D^2 h - A| \le C \rho \eps $ in $B_\rho$ imply $$D_{\xi\xi}^2h \le - c_2(\beta) \eps \quad \mbox{ in} \quad B_\rho,$$ 
for some unit direction $\xi$. 
By \eqref{OutsideCylinder2} a similar estimate holds for $u$ outside $\mathcal C_\eta$. This and $u(0)=0$, $u \ge 0$, show that $\|u-q\|_{L^\infty(B_\rho)} \ge c \eps$ for any quadratic polynomial $q$.  

Next we apply Lemma \ref{L1} inductively and obtain \eqref{pmrm} with $\eps_m \ge c(\eps,\beta) 4^{-m}= c(\eps,\beta) (\rho^m)^\alpha$, for some $\alpha$, which shows that $u$ is not better than $C^{2,\alpha}$ at the origin.

On the other hand the value of $\rho$ can be reset to be smaller and smaller after a finite number of steps. This is because the rescaled value of $\kappa_0$ tends to infinity as $\eps_m \to 0$ which means that we may take the parameters $\delta$, $\eta$, $\rho$ to be smaller and smaller. 
Thus $u$ is not $C^{2,\alpha}$ at the origin for any $\alpha>0$, and Claim 1 is proved. 

\

{\it Claim 2: $\Sigma \cap B_\rho$ is in a $\sigma(\delta)$-neighborhood of a subspace of codimension at least $k+1$, and $\sigma(\delta) \to 0$ as $\delta \to 0$.} 

Under our assumption $\lambda_1(D^2p)\ge\lambda_2(D^2p)\ge\dots\ge\lambda_k(D^2p)\ge\delta^{-4}\eps$, we already have $\Sigma\subset\{|x'|\le\delta^2\}$. The goal is to exploit the negative eigenvalue of $D^2(\hh-\hO)$ in alternative (2) to get estimates in one of the $x''$-directions.

By hypothesis $D^2(\hat h - \hat O)(0)$ has an eigenvalue less than $-c_1(\beta)$. Since
$$D_{x'x'}^2 (\hat h - \hat O) \ge \delta^{-4}, \quad |D_{x'x''}^2 (\hat h - \hat O)| \le C,$$
we can find a unit direction $\xi$ belonging to the $x''$ subspace, say $\xi=e_{k+1}$, such $D_{\xi\xi}^2(\hat h - \hat O) \le - \frac 12 
c_1(\beta)$. By Taylor's expansion as in Step 3, we obtain that
$$(\hat h - \hat O)((\minx, 0) + t \xi) \le C \delta - c(\beta) t^2, \quad \quad |t| \le \rho.$$
Then, by Remark \ref{r22}, we conclude $\Sigma \cap B_\rho \subset \{ |x_{k+1}| \le C(\beta) \delta^{1/2}=:\sigma(\delta)\}$. \end{proof}

\section{Results for the thin obstacle problem}\label{Thin}
In this section, we collect some results about the thin obstacle problem which are useful in the analysis of the top stratum $\Sigma^{d-1}$. 

Assume $v$ solves the thin obstacle problem in $B_1$,
\begin{equation}\label{ThinObstacle}
\begin{cases}
\Delta v\le 0 &\text{ in $B_1$,}\\ \Delta v=0 &\text{ in $B_1\cap(\{v>0\}\cup\{x_1\neq 0\})$,}\\ v\ge 0 &\text{ along $\{x_1=0\}$}.
\end{cases}
\end{equation} 
Here $x_1$ denotes the first coordinate function of $\R^d$. 

It is well known that $v$ is locally Lipschitz, and in fact $v \in C^{1,1/2}$ when restricted to each half-space $\{x_1 \ge 0\}$ and $\{x_1 \le 0\}$. 
An important tool is the Almgren frequency formula which states that if $v$ does not vanish identically then
$$ \lambda (r):=\frac{r\int_{B_r}|\nabla v|^2 dx}{\int_{\partial B_r}v^2 dx} \quad \quad \mbox{is monotone increasing in $r$,}$$
and $\lambda (r)$ is constant in $r$ if and only if $v$ is a homogeneous solution to \eqref{ThinObstacle}.

We define the frequency of $v$ at $0$ as
$$\lambda:=\lim_{r \to 0^+} \lambda (r).$$
The rescalings of $v$ at $0$ that fix the $L^2$ norm on $\partial B_1$, 
$$ v_r(x):= \frac{r^\frac{d-1}{2}}{\|v\|_{L^2(\partial B_r)}} \, v(rx) $$
converge along subsequences to a global $\lambda$-homogeneous  solution, which we refer to as a blow-up profile of $v$ at $0$.   
In particular for each small $r$, there exists a a global $\lambda$-homogeneous global solution $V_r$ such that
\begin{equation}\label{exp}
\|v - V_r\|_{L^\infty(B_r)} \le \delta(r) \|V_r\|_{L^\infty (B_r)}, \quad \quad \delta(r) \to 0 \quad \mbox{as $r \to 0$,} 
\end{equation}
and
$$r^{\lambda -\alpha} \ge \|V_r\|_{L^\infty(B_r)} \ge r^{\lambda +\alpha},$$
for any fixed $\alpha>0$ and all $r$ sufficiently small.

Athanasopoulos, Caffarelli and Salsa in \cite{ACS} showed that the only possible values for the frequency $\lambda$ that are less than 2 are $\lambda=0$ (when $v(0)>0$), $
\lambda=1$ (when $0$ is interior to $\{v=0\}$ in $\{x_1=0\}$), or $\lambda=\frac 32$. In this last case the uniqueness of the blow-up profile was 
established as well, which means that the expansion \eqref{exp} holds for a fixed non-zero homogeneous solution $V$ in place of $V_r$. 

Concerning higher frequencies, when $\lambda$ is 
an even integer, Garofalo and Petrosyan in \cite{GP} characterized all possible $\lambda$-homogeneous solutions as harmonic polynomials. They also proved the uniqueness of blow-ups for these values of $\lambda$. Colombo, 
Spolaor and Velichkov in \cite{CSV2} sharpened these results through a log-epiperimetric inequality, and obtained a frequency gap near the even 
integers. 

More recently, Figalli, Serra and Ros-Oton in \cite{FSR} characterized the $\lambda$-homogeneous solutions when $\lambda$ is 
an odd integer, and proved the uniqueness of the blow-ups for these values of $\lambda$.

Below we give a short proof that establishes the frequency gap near all integers.

\begin{thm}[Frequency gap near integers]\label{FG} Suppose there is a non-trivial solution to the thin obstacle problem \eqref{ThinObstacle} with frequency $\lambda.$

For each $m \in \mathbb N$, there exits $\alpha_m>0$ small, depending only on the dimension $d$ and $m$ so that $$ \lambda \notin (m-\alpha_m,m+\alpha_m) \setminus\{m\}.$$ 
\end{thm}
\begin{proof}
{\it Case 1: $m=2k$ is even.}

We first point out that any $2k$-homogeneous solution to \eqref{ThinObstacle} must be a harmonic polynomial. This is already known by Garofalo-Petrosyan \cite{GP}. 

To see this, let $v$ be a $2k$-homogeneous solution, and let $p$ be any $2k$-homogeneous harmonic function. Then 
\begin{equation}\label{pdm0}
\int_{B_1}p(-\triangle v)= \int_{B_1}v \triangle p - p \triangle v= \int_{\partial B_1} v p_\nu - p v_\nu=0.
\end{equation} The last equality follows from the homogeneity of the functions. 

Apply this with $p=P_I$, the $2k$-homogeneous harmonic function with $P_I(0,x')=|x'|^{2k}$ in $\{x_1=0\}$, we have $\int_{B_1}P_I(-\triangle v)=0.$ Note that $-\triangle v$ is a non-negative measure supported on $\{x_1=0\}$, this implies $$\Delta v=0.$$

Now suppose, on the contrary, that there is a sequence of non-trivial solutions to \eqref{ThinObstacle}, denoted by $v_j$, that are $\lambda_j$-homogeneous with $\lambda_j\neq 2k$ but $\lambda_j\to 2k.$

Up to a normalization, we assume $\|v_j\|_{L^2(B_1)}=1$. Then up to a subsequence, we have $v_j\to v$ locally uniformly in $B_1$, where $v$ is a $2k$-homogeneous solution with $\|v\|_{L^2(B_1)}=1$. The convergence is uniform in $C^1$ if we restrict the domain to $B_1\cap\{x_1\ge 0\}$ or $B_1\cap\{x_1\le 0\}$. 

A similar computation as in \eqref{pdm0} gives:
\begin{equation}\label{EvenContradiction}\int_{B_1}(v\pm\delta P_I)(-\triangle v_j)=(2k-\lambda_j)\int_{\partial B_1}(v\pm\delta P_I)v_j.\end{equation}

Locally uniform convergence of $v_j\to v$  and homogeneity of the functions imply that $\int_{\partial B_1}vv_j\ge c_d>0$ for all large $j$. Consequently, fix $\delta>0$ small, we have $\int_{\partial B_1}(v\pm\delta P_I)v_j>0$ for all large $j$. In particular, regardless of the sign in front of $\delta,$ the right-hand side of \eqref{EvenContradiction} has the same sign as $(2k-\lambda_j)$.

On the other hand, by the locally uniform convergence of $v_j\to v$ and homogeneity of the functions, we have that the  support of $-\triangle v_j$ is contained in $\{v<\frac\delta 2 P_I\}$ for large $j$. Consequently, in \eqref{EvenContradiction}, the left-hand side is non-negative if we choose the positive sign in front of $\delta$, and non-positive otherwise. This is a contradiction.

{\it Case 2: $m=2k+1$ is odd. }

We first point out that any $(2k+1)$-homogeneous solution to \eqref{ThinObstacle} must vanish along the hyperplane $\{x_1=0\}$. This is already known as in \cite{FSR}. We sketch the proof here for completeness. 

To see this, let $v$ be a $(2k+1)$-homogeneous solution. Let $Q_I$ denote the $(2k+1)$-homogeneous function that is even with respect to $\{x_1=0\}$, harmonic in $\{x_1\neq 0\}$, and satisfies $Q_I(0,x')=0$ and $\triangle Q_I=-|x'|^{2k}d\mathcal{H}^{d-1}|_{\{x_1=0\}}.$

A similar computation as in \eqref{pdm0} gives
$$\int_{B_1}v(-\triangle Q_I)=0.$$With $v\ge 0$ on $\{x_1=0\}$, this forces $v=0$ on $\{x_1=0\}.$

Suppose  that there is a sequence of non-trivial solutions to \eqref{ThinObstacle}, denoted by $v_j$, that are $\lambda_j$-homogeneous with $\lambda_j\neq 2k+1$ but $\lambda_j\to 2k+1.$

Similar to the previous case, we have $v_j\to v$, where  $v$ is a $(2k+1)$-homogeneous solution, and that $\|v_j\|_{L^2(B_1)}=\|v\|_{L^2(B_1)}=1$. A similar computation as in \eqref{pdm0} gives 
$$\int_{B_1}v_j(\pm\delta\triangle Q_I-\triangle v)=(\lambda_j-2k-1)\int_{\partial B_1}v_j(\pm\delta Q_I-v).$$

Similar to the previous case, when $\delta$ is small, the right-hand side has the same sign as $(\lambda_j-2k-1)$, regardless of the sign in front of $\delta.$

On the other hand, for large $j$, with locally uniform convergence of $v_j\to v$ as well as the homogeneity of the functions, we have $\{v_j\neq 0\}\cap\{x_1=0\}\subset\{\triangle v\ge \frac\delta 2 \triangle Q_I\}$. Thus the left-hand side is non-positive if the sign in front of $\delta$ is positive, and non-negative otherwise. This is a contradiction. \end{proof}

We will use Theorem \ref{FG} only for $m=2,3$. As a consequence we obtain the following result.

\begin{lem}\label{l31}
Suppose that $v$ solves \eqref{ThinObstacle} and that it is $\eta$-approximated by a $\lambda$-homogeneous function $W$ with $\lambda \le 3-\alpha_3$ (with $\alpha_3$ in Theorem \ref{FG}). 
\begin{equation}\label{v-V}
\|v-W\|_{L^\infty(B_1)} \le \eta \|W\|_{L^\infty(B_1)}.
\end{equation}
If $\eta \le c$ with $c$ depending only on the dimension $d$, then $$\lambda_{x_0} \le 3-\alpha_3, \quad \mbox{in $B_{1/2}$,}$$ where $\lambda_{x_0}$ denotes the frequency at a point $x_0\in\partial \{v>0\}\cap\{x_1=0\}$. 
\end{lem}
We remark that $W$ is not assumed to be a homogeneous solution to the thin obstacle problem.
\begin{proof}
It follows by compactness. 
If $v_k$, $W_k$, $\sigma_k$ satisfy the hypotheses with $\sigma_k \to 0$ and $\|W_k\|_{L^\infty}=1$, 
then we can extract a subsequence such that $v_k \to \bar v$, $W_k \to \bar v$ in $B_{3/4}$ with $\bar v$ a homogeneous solution of degree $\bar 
\lambda \le 3-\alpha_3$. Then $\lambda^{\bar v}_{x_0}(1/8) \le 3-\alpha_3$ which means $\lambda^{v_k}_{x_0} \le \lambda^{v_k}_{x_0}(1/8) \le 3- \frac 12 
\alpha_3$, for all large $k$ and the conclusion follows by Theorem \ref{FG}.
\end{proof}

We will use estimates of the type \eqref{v-V} to express that the ``frequency" of $v$ is less than 3. This is convenient for perturbations of the thin obstacle problem, where the monotonicity formula might not apply. We make the following definition for functions that do not necessarily solve \eqref{ThinObstacle}.

\begin{defi}\label{aeta}
We say that $$w \in \mathcal A (\eta, B_r)$$
if $w$ can be $\eta$-approximated in $B_r$ by a $\lambda$-homogeneous function $W$ with $\lambda \le 3-\alpha_3$,
$$\|w-W\|_{L^\infty(B_r)} \le \eta \, \, \|W\|_{L^\infty(B_r)}.$$

\end{defi}

\begin{lem}\label{n13} Assume that $v$ solves \eqref{ThinObstacle} and $v \in \mathcal A (c, B_1)$ with $c$ as in Lemma \ref{l31}. For any $\eta>0$, there exists $c(\eta)$ depending on $\eta$ and $d$ such that 
$$  v \in \mathcal A(\eta,B_r), \quad \mbox{for some $r \in [c(\eta), 1/2]$.}$$

\end{lem}

\begin{proof}
The proof follows by compactness. Assume that $v_k$ satisfies the hypothesis, and say $\|v_k\|_{H^1(B_1)} =1$, but the conclusion does not hold for any $r \in [1/k,1/2]$. Then, we can extract a convergent subsequence to a limit solution $\bar v$, for which we can find $r$ such that $\bar v \in \mathcal A (\eta/2, B_r)$. This implies that $v_k \in \mathcal A(\eta,B_r)$ for all large $k$, and we reached a contradiction.
\end{proof}

Finally, we may use the result of Focardi and Spadaro \cite{FoS} on the Haudorff dimension of the set of free boundary with frequency 
$\lambda \notin \cup_{m \in \mathbb N}\{m, 2m-\frac 12\}$, i.e.
$$ \Gamma:=\left\{x \in \{v=0\} \cap \{x_1=0\}| \quad \lambda_x \ne m, \quad \lambda_x \ne 2m-\frac 12, \quad \forall m \in \mathbb N \right\} $$

\begin{lem}[Covering of $\Gamma$] \label{cover}Assume that $v$ solves \eqref{ThinObstacle} and $ v \in \mathcal A(\eta,B_{2})$ for $\eta$ small. 
For any $\mu > d-3$, there is a finite cover of $\Gamma$ such that 
$$\Gamma \cap B_1 \subset \, \, \bigcup B_{r_i}(x_i) \quad \quad \mbox{with} \quad \sum r_i^\mu \le \frac 12, \quad \quad r_i \ge c_1,$$
and $$ v \in \mathcal A(\eta,B_{2r_i}(x_i)) \quad \mbox{for each $i$.}$$
Here $c_1=c_1(\eta,\mu)$ depends on $\eta$, $\mu$ and $d$.
\end{lem}

\begin{proof}
The proof is again by compactness. Assume that $v_k$ satisfies the hypothesis, and say $\|v_k\|_{H^1(B_1)} =1$, but the conclusion does not hold with $c_1=1/k$. Then, we can extract a convergent subsequence to a limit solution $\bar v \in \mathcal A(\eta,B_2)$. By Lemma \ref{l31} the corresponding set $\Gamma_{\bar v}$ consists of those points for $\bar v$ which have frequency $\lambda_x \in (2,3)$, since the frequencies $\le 2$ must belong to the set $\{0,1, \frac 3 2,2\}$, see \cite{ACS}. By Theorem \ref{FG}, $\Gamma_{\bar v}$ is a closed set in $\overline{B}_1$. 
Since the dimension of $\Gamma_{\bar v}$ is $d-3$, see \cite{FoS}, given any $\sigma>0$ we can find a finite cover of $\Gamma_{\bar v} \cap B_1$ by balls $B_{s_i}(x_i)$ with $\sum s_i^\mu \le \sigma$. According to Lemma \ref{n13} each ball, can be enlarged by at most a $C(\eta)$ factor such that
$$ \bar v \in \mathcal A (\eta /2, 2 r_i), \quad \mbox{with} \quad s_i \le r_i \le C(\eta) s_i.$$
Then
$$\sum r_i^\mu \le C(\eta)^\mu \sum s_i^\mu \le \sigma C(\eta)^\mu \le \frac 12,$$
provided that $\sigma$ is chosen small. 

On the other hand $\Gamma_k \cap B_1 \subset \cup B_{r_i}(x_i)$ for all large $k$. 
Indeed, if $x_0$ is a point outside the union of balls, then $\bar N_{x_0}(r) \le 2 + \frac 12 \alpha_2$ for some small $r$, 
which shows that $\Gamma_k$ cannot intersect a small neighborhood of $x_0$ for all large $k$.

In conclusion $v_k$ satisfies the conclusion for all large $k$ and we reached a contradiction.
\end{proof}

\begin{rem}\label{R3.9}
The proof shows the conclusion can be replaced by $v \in \mathcal A (\kappa \eta, K r_i)$ for any constants $\kappa$ small, $K$ large 
provided that $c_1$ depends on these constants as well. 
In particular, if $v$ solve the thin obstacle problem for a constant coefficient linear operator 
$L_A v= tr A D^2v$ of ellipticity $\Lambda$, which can be reduced to the case of $\triangle$ after an affine deformation, 
then we can choose $\eta$ small depending on $\Lambda$ and $d$, so that Lemma \ref{l31} applies, 
and then $c_1=c_1(\mu,\Lambda,d)$ small so that Lemma \ref{cover} holds with $v \in \mathcal A (\eta/8, 2 r_i)$.
\end{rem}

\section{A refinement of Lemma \ref{QuadraticApproxTopStratum}}\label{ke1}

In this section we give a more precise version of Lemma \ref{QuadraticApproxTopStratum} and characterization of the top stratum $\Sigma^{d-1}$. 
In view of Theorem \ref{T1} each stratum $\Sigma^k$ is well defined. 
Indeed, if one of the blow-up quadratic polynomials at $0 \in \Sigma$ (see Proposition \ref{C11}), say $p$, 
satisfies $\lambda_2(D^2 p) >0$, then Theorem \ref{T1} applies, 
and we obtain that $p$ is the unique blow-up profile and $0 \in \Sigma^k$ with $k= dim  \, ker D^2 p$, $ k \le d-2$. 
Then we define the top stratum as $$\Sigma^{d-1}:= \Sigma \setminus \cup_{k=0}^{d-2} \Sigma^k.$$ 
Near a point in $\Sigma^{d-1}$, we use one-dimensional homogeneous quadratic polynomials $p$ 
with $D^2p \ge 0$, $\lambda_2(D^2p)=0$, $F(D^2p)=1$ to approximate the solution. We define the space of such polynomials as
$$\mathcal Q_0:=\left \{p:\quad p(x)=\frac{1}{2}x^TAx, \quad A= \gamma(e)\,  e \otimes e, \quad F(A)=1 \right \}.$$
Similar to Definition \ref{ApproxClass}, we define by $\mathcal S_0(p,\eps,r)$ the class of solutions to \eqref{OP} which are $\eps$-approximated by a quadratic polynomial $p \in \mathcal Q_0$ in $B_r$.

\begin{defi}\label{ApproxClass0}
Given $\eps,r\in(0,1)$ and $p\in\mathcal Q_0$, we say that
$$u\in\mathcal{S}_0(p,\eps,r)$$
  if 
  \begin{equation}\label{u-per2}
  u \text{ solves \eqref{OP} in $B_r$, and } \quad |u-p|\le\eps r^2 \text{ in $B_r$}.
  \end{equation} 
  \end{defi}

We remark that in Definition \ref{ApproxClass0}, we no longer require the second derivative bound $D^2 u \ge - c_0 \eps$ 
as in Definition \ref{ApproxClass}, which played an important role in Section \ref{kg2}. 
However, from \eqref{u-per2} we can always deduce a bound of the type $D^2 u \ge - C \eps$ in $B_{r/2}$, 
see Lemma \ref{L222} below.

We state the main lemma of this section, which is a dichotomy for the top stratum $\Sigma^{d-1}$.

\begin{lem}\label{L2} 
Assume that
$$0 \in \Sigma^{d-1}, \quad \quad u\in\mathcal{S}_0(p,\eps,1).$$
Given $\beta <1$, there are constants $\bar{\eps}$, $\rho$ small, depending on $\beta$ and the universal constants such that if $\eps<\bar{\eps}$, then 
$$ u\in\mathcal{S}_0(p',\eps', r), \quad \mbox {for some $r \in [\rho,\frac 12]$,}$$
and either
\begin{enumerate}
\item{ $$\eps'=\eps \, r^\beta,$$}
\item {or $\eps'=r^{-2}\|u-p'\|_{L^\infty(B_r)}$ and (see Definition \ref{aeta}) 
$$ \frac{1}{\eps'}(u-p') \in \mathcal A(\eta, B_{r}), \quad \quad \eps r^ {\alpha_0} \ge \eps'  \ge \eps r ^{\alpha_1},$$ 
for some constants $0<\alpha_0<\alpha_1<1$ depending on the dimension $d$, and $\eta$ small universal.}
\end{enumerate}

\end{lem}

We can iterate Lemma \ref{L2} indefinitely. Lemma \ref{L22} below shows that if end up in case (2) during the iteration, 
then we remain in case (2) and $u$ cannot be $C^{2,\alpha}$ at the origin for $\alpha$ close to 1. 
We define such a point to be anomalous  for $\Sigma^{d-1}$.

\begin{defi}\label{AN2}
We say that $x_0 \in \Sigma^{d-1}$ is anomalous and write $x_0 \in \Sigma^{d-1}_a$ if 
$$ \|u - \frac 12 (x-x_0)^T D^2 u(x_0) (x-x_0)\|_{L^\infty(B_r(x_0))} \ge r ^{2+\alpha} $$
for some $\alpha<1$ and all $r$ small. We denote
$$\Sigma^{d-1}_g:=\Sigma^{d-1} \setminus \Sigma^{d-1}_a.$$
\end{defi}  

Next we describe the iteration step of case (2) in Lemma \ref{L2}.
 
\begin{lem}\label{L22} 
Assume that 
\begin{equation}\label{s0p}
u\in\mathcal{S}_0(p,\eps,2), \quad \frac{1}{\eps} (u-p) \in \mathcal A(\eta,B_2).
\end{equation}

Then

a) The conclusion (2) holds at any point in $\Sigma^{d-1} \cap B_1$ and $\Sigma^{d-1}=\Sigma_a^{d-1}$ in $B_1$.

b) Fix $\mu> d-3$. Then, if $\eps \le \bar\eps(\mu)$ small, we can find a cover of $\Sigma^{d-1} \cap B_1$ with $B_{r_i}(x_i)$ such that
$$ \Sigma^{d-1} \cap B_1 \subset \, \, \cup B_{r_i}(x_i), \quad \quad \sum r_i^\mu \le \frac 12, \quad r_i \ge c(\mu),  $$
$$u\in\mathcal{S}_0(p_i,\eps_i, 2 r_i), \quad \frac{1}{\eps_i} (u-p_i) \in \mathcal A(\eta,B_{2r_i}).$$
Moreover $\Sigma^{d-1}_a = \emptyset$ in dimension $d=2$ , and $\Sigma^{d-1}_a$ is finite if $d=3$.

\end{lem}

We combine Lemmas \ref{L2} and \ref{L22} and obtain the following characterization of the set $\Sigma^{d-1}$.
\begin{thm}\label{T2}
Assume that $\eps \le \bar \eps_1$, 
$$u\in\mathcal S_0(p,\eps,1),$$
with $\bar{\eps}_1$ a small, universal constants. Then in $B_{1/2}$ we have

a) $u$ is $C^{2,\alpha_0}$ on $\Sigma^{d-1}$: 
for each $x_0 \in \Sigma^{d-1}$, $\exists \, q_{x_0}$ quadratic polynomial, with 
$$|u-q_{x_0}|(x)\le C \eps |x-x_0|^{2+\alpha_0}.$$

b) $u$ is uniform $C^{2,\beta}$ for any $\beta<1$ on the non-anomalous set $\Sigma^{d-1}_g$: 
$$|u-q_{x_0}|(x)\le C(\beta) \eps |x-x_0|^{2+\beta}, \quad \quad \mbox{if} \quad x_0 \in \Sigma^{d-1}_g,$$
and the anomalous set $\Sigma_a^{d-1}$ is open in $\Sigma^{d-1}$ and has Haussdorff dimension $d-3$. In particular in dimension $d=2$, $
\Sigma^1_a = \emptyset$, and in dimension $d=3$, $\Sigma^2_a$ is finite.

\end{thm}

Part a) was obtained in \cite{SY}, while part b) is a consequence of the results of this section.

The proofs of Lemma \ref{L2} and \ref{L22} follow the same strategy as Lemma \ref{L1} in Section \ref{kg2}. 
The difference is that now the rescaled error 
\begin{equation}\label{hau-p}
\hat u := \frac{1}{\eps}(u-p)
\end{equation}
is well approximated by a solution $v$ to the thin obstacle problem involving $L_p$, the linearized operator of $F$ at $D^2 p$, which has constant 
coefficients. 
Then the dichotomy (1) or (2) is dictated by the frequency $\lambda_0$ of $v$ at $0$, whether or not $\lambda_0 \ge 3$. 
We will show that $\lambda_0 >2$ as follows: 

a) if $\lambda_0 =1$ then $0 \in Reg$; 

b) if $\lambda_0 = 3/2$ then $0 \in Reg$; 

c) if $\lambda_0=2$ then $0 \in \Sigma^k$ for some $k \le d-2$. 

Before we proceed with the proofs of the main lemmas we provide a lower bound for $D^2 u$.

\begin{lem}\label{L222} Assume that $u$ solves \eqref{OP} and
$$|u-p| \le \eps \quad \mbox{in $B_1$,} $$
for some convex quadratic polynomial $p$ with $F(D^2p)=1$. Then
$$D^2 u \ge - C \eps \quad \mbox{in $B_{3/4}$,} $$
for some $C$ universal.
\end{lem}

\begin{proof}
Assume that $p=\sum \frac 12 a_i \, x_i^2$, with $a_1 \ge a_2 \ge ..\ge a_d \ge 0$. Then $\{u=0\} \subset \{ |x_1| \le C \eps^{1/2}\}$, and by Proposition \ref{EstimateForDifference} we find
$$ |D^2 (u-p)| \le C \eps |x_1|^{-2}  \quad \mbox{in} \quad B_{7/8} \cap \{ |x_1| \ge C \eps^{1/2}\}.$$
We use that $D^2p \ge 0$, together with $D^2 u \ge - C_d I$ in the strip $\{ |x_1| \le C \eps^{1/2}\}$ and conclude that $w=u_{ee}^-$ satisfies
$$ L_u w \ge 0, \quad w \ge 0, \quad \quad \|w\|_{L^p(B_{7/8})} dx \le C \eps,$$
for $p=\frac 13$.
By weak Harnack inequality we obtain $|w| \le C \eps$ in $B_{1/2}$ and the lemma is proved.
\end{proof}

Without loss of generality we may assume that after an affine transformation of bounded norm, and a rotation we satisfy
\begin{equation}\label{Id}
p= \frac 12 a_1 x_1^2, \quad \quad D F(D^2 p) =I, \quad \quad L_p w =\triangle w. 
\end{equation}

\begin{lem}\label{UniformLipschitz}
Let $u$ and $p$ be as in Lemma \ref{L2} and $\hat u$ as in \eqref{hau-p}. Then
$$|\nabla \hat u |\le C \text{ in $B_{1/2}$.}$$ 
\end{lem}

\begin{proof} By Lemma \ref{L222} we know that in $B_{3/2}$, $\hat u_{ee} \ge -C$ for all unit directions $e \perp e_1$. 
Since $\|\hat u\|_{L^\infty} \le 1$ we obtain $|D_e u| \le C$ in $B_{1/2}$. 

On the other hand $\triangle \hat u = L_p \hat u \le 0$, thus we also have $\hu_{11}\le C$ in $B_{3/4}$, which gives $|D_1 u| \le C$ in $B_{1/2}$.\end{proof}  

This lemma provides us with enough compactness for the family of normalized solutions. 

\begin{lem}\label{LimitingProblem}
Let $F_j$ be a sequence of operators satisfying \eqref{FirstAssumption}-\eqref{Ellipticity}, and $u_j$, $p_j$, $\eps_j$ satisfying the assumptions of 
Lemma \ref{L2}, and \eqref{Id}, with $\eps_j \to 0$. 
Then up to a subsequence, the normalized solution 
$$\hat{u}_j=\frac{1}{\eps_j}(u_j-p_j)$$ converges locally uniformly in $B_1$ to a solution $v$ to the thin obstacle problem \eqref{ThinObstacle}.
\end{lem}

\begin{proof}
Lemma \ref{UniformLipschitz} gives locally uniform $C^{0,1}$ bound on the family $\{\hat{u}_j\}$. Consequently, up to a subsequence they converge to some $v$ and, 
$$G_j(D^2\hat{u}_j)=\frac{1}{\eps_j}(\chi_{\{u_j>0\}}-1)=-\frac{1}{\eps_j}\chi_{\{u_j=0\}},$$
with
$$G_j(M):=\frac{1}{\eps_j}(F_j(\eps_jM+D^2p_j)-F_j(D^2p_j)).$$ 
 By uniform $C^{1,\alpha_F}$ estimate on the family $\{F_j\}$,  up to a subsequence $G_j$ locally uniformly converges to $\triangle$, and the result 
 easily follows. 

 \end{proof}

Now we give the proof of Lemma \ref{L2}, which follows the steps of Lemma 4.1 in \cite{SY}. The main difference is that now we include the discussion on the frequency 2 case, see Step 3 below, which was not done in \cite{SY}. 

\begin{proof}[Proof of Lemma \ref{L2}]

Suppose by contradiction that for a sequence of $\eps_j, \rho_j \to 0,$ a sequence of operators $F_j$ and a sequence of solutions $u_j$ to \eqref{OP} with these operators such that $$u_j\in\mathcal{S}_0(p_j,
\eps_j,1), \quad \quad 0\in\Sigma^{d-1}(u_j),$$ and \eqref{Id} holds,  
but the conclusion of Lemma \ref{L2} does not hold for $u_j$, with $\eta$ as in Remark \ref{R3.9}.
 Lemma \ref{LimitingProblem} shows that up to a subsequence, 
 $$\hat{u}_j\to v \text{ locally uniformly in $B_{1}$,}$$ 
 where $v \in C^{0,1}_{loc}$ solves the thin obstacle problem \eqref{ThinObstacle}.
 
 Moreover, $u_j(0)=0$ for all $j$ implies $v(0)=0.$ Denote by $\lambda_0$ the frequency of $v$ at $0$, and by $\lambda_0^*$ the frequency of $v^{e}$, the even part of $v$ with respect to the $x_1$ variable,
 $$v^{e}(x):= \frac 12 (v(x_1,x_2,..,x_n) + v (-x_1,x_2,..,x_n)).$$ 
 In Steps 1-2 which are identical with \cite{SY}, we show that $\lambda_0 \ge 2$. In Step 3 we prove that $\lambda_0^*>2$. Then in Steps 4-5 we establish the conclusions (1)-(2) for $u_j$, depending on whether or not $\lambda_0^* \ge 3$. 
We only sketch the first 2 steps, the details can be found in Lemma 4.1 in \cite{SY}. 

\

\textit{Step 1: $\nabla v(0)=0$.}

\

Since $v(0)=0$ we have the expansion 
$$v=a_+x_1^++a_-x_1^ +  + o(|x|) \quad \mbox{ as $x\to 0.$}$$

First we claim that $a_\pm \le 0$. 
Indeed, if say $a_+ >0,$ then we can use the uniform convergence of the $u_j$ to $v$ 
and an explicit barrier to show that $u_j(0)>0$ for all large $j$, contradiction.

Then we claim that $a_\pm$ cannot be negative. 
If say $a_+<0,$ then we can use a barrier to prove that $\{u_j=0\}$ contains a small open ball around a point $t e_1$ for some $t$ small. On the other 
hand Lemma \ref{Convexity} implies that $u_{ee}\ge 0$ in $B_{1/2}$ for all unit direction $e$ close to $e_1$. This means that $\{u_j=0\}$ 
contains an open cone with vertex at $0$, hence $0 \in Reg(u_j)$, contradiction.

\

\textit{Step 2: $\lambda_0 \ge 2$.}

\

In view of Step 1 we only need to show that $\lambda_0 \ne \frac 32$. Otherwise, $v$ has an expansion
$$ v = a  Re(z^{\frac 32}) + o(|x|^{\frac 32}), \quad a>0,  $$
where $z$ represent the complex number in a 2d plane generated by unit directions $\nu$ and $e_1$ for some $\nu \perp e_1$. 
Then we fix some $r>0$ so that $D_\nu v> 0$ in $B_r\cap\{x_1\neq0\}$, and pick $\sigma$ small 
such that $D_\nu v \ge 2 \sigma $ in $\{|x_1| \ge c' r\} \supset \{u_j > c r^2\}$ for all large $j$. 

Since for $e \perp e_1$ we have $\frac{1}{\eps_j} D_e u_j=D_e \hat u_j \to D_e v$ uniformly on compact sets of $B_r \setminus \{x_1=0\}$, and by Lemma \ref{UniformLipschitz}, $D \hat u_j \ge -C$, we see that Remark \ref{Monotonicity2} applies for $u_j$ and such unit directions $e$ close to $\nu$. We obtain $D_e u_j \ge 0$ and contradict $0 \in \Sigma(u_j)$. 

\

\textit{Step 3: $\lambda_0^*>2$.}

\

We decompose $v=v^{o}+ v^{e}$ in the odd and even part with respect to $x_1$ variable. Then $v^{o}$ is a harmonic function which vanishes on $x_1=0$, and by Step 1
$$ v^{o}=\sum_{i>1} b_i x_1 x_i + O (|x|^3).$$
The even part $v^e$ still solves the thin obstacle problem \eqref{ThinObstacle}, and by Steps 1-2, its frequency at $0$, $\lambda_0^* \ge 2$. 

Let us assume by contradiction that $\lambda_0^*=2$. Then $v$ has an expansion
$$v=\frac 12 x \cdot A x + o(|x|^2) $$
with $tr \, A=0$, $e\cdot Ae\ge 0$ for all $e\in\Sph\cap\{x_1=0\},$ 
and with strict inequality for some unit direction say $e_2$, $e_2\cdot Ae_2=a_2 > 0$.
Consequently, 
$$|v-\frac{1}{2}x\cdot Ax|<\delta(r) \, r^2 \text{ in $B_r$, with $\delta(r) \to 0$ as $r \to 0$.}$$
The uniform convergence of $\hat{u}_j\to v$ gives for large $j$ (we drop the subindex $j$ for simplicity of notation)
 \begin{equation}\label{QuadApprox2}
|u-p -\eps\frac{1}{2}x\cdot Ax|<2\delta\eps r^2 \text{ in $B_r$.}
\end{equation}  

By Cauchy-Schwarz inequality there is a constant $C_A$, depending on $|A|$, so that 
 \begin{equation}\label{66}D^2p+ \eps A +C_A\eps^2 I \ge c \, e_1\otimes e_1 + \eps \, a_2 \, \, e_2 \otimes e_2,\end{equation}
 for some $c$ universal. The hypotheses on $F$, \eqref{Id}, and $tr \, A=0$ imply that  
 $$|F(D^2p+ A\eps+C_A\eps^2I)-1|\le C_A'\eps^{1+\alpha_F}$$ by assumption \eqref{SecondAssumption}.
Consequently, there is $t\in [-C_A'',C_A'']$ such that the polynomial $$p'(x)=p+\frac{1}{2}\eps x\cdot Ax+\frac{1}{2}C_A\eps^2|x|^2+\frac{1}{2}t\eps^{1+\alpha_F}x_1^2$$ satisfies $F(D^2p')=1$ and, by \eqref{66}, we have (for all $\eps$ small)
$$ D^2p'=D^2p + \eps A +C_A\eps^2I+t\eps^{1+\alpha_F}e_1\otimes e_1 \ge \frac c 2 \, e_1\otimes e_1 + \eps \, a_2 \, \, e_2 \otimes e_2 \ge 0,$$  
and $\lambda_2 (D^2 p') \ge a_2 \eps$. 
Finally \eqref{QuadApprox2} implies that 
\begin{equation}\label{u-p'}
|u-p'| \le 2\delta\eps r^2+C_A''\eps^{1+\alpha_F}r^2 \le 3\delta\eps r^2 \quad \quad \mbox{in $B_r$.}
\end{equation}
By Lemma \ref{L222} this implies that $D^2 u \ge - C \delta \eps$ in $B_r$ for some $C$ large universal. Thus $u$ satisfies the hypotheses of Lemma \ref{L1}
$$ u \in \mathcal S (p',\eps',r) \quad \quad \mbox{with} \quad \eps':= -c_0^{-1} C \delta \eps, \quad \mbox{and} \quad \lambda_2(D^2 p') \ge a_2 \eps \ge \kappa_0 \eps',$$
provided that we choose $r$ small (depending on $a_2$) such that $\delta=\delta(r) \le a_2 C'$ for some $C'$ large universal. In conclusion $0 \in \Sigma^k$ for some $k \le d-2$, and we reached a contradiction.

\

\textit{Step 4: If $\lambda_0^* \ge 3$ then (1) holds.}

\

Then, by using the behavior of $v^o$ and $v^e$ near $0$ we find that
$$v=\sum_{i>1} b_i x_1 x_i  + o(|x|^{2+\beta}),$$
and the computations of Step 3 apply with the corresponding $A$, and with $\delta$ replaced by $\delta r^\beta$. Notice that in this case we can 
construct $p'$ such 
that $D^2 p'$ has rank 1, i.e. it is a multiple of $e_1' \otimes e_1'$ where $e_1'$ is the unit direction of $e_1 + \eps \sum _{i>2}b_i/a_1 e_i$. We achieve 
this by first completing the square in $p + \eps \sum_{i>1} b_i x_i x_1$ and then by adding a multiple $ t  \eps^{1+\alpha_F}e_1' \otimes e_1'$ so that 
$F(D^2p')=1$.  As in Step 3 we obtain \eqref{u-p'} thus $$ u \in \mathcal S_0(p',\eps',r), \quad \eps' \le \eps r^\beta,$$
and (1) holds for all large $j$, and we reached a contradiction.

\

\textit{Step 5: If $\lambda_0^* \in (2,3)$ then (2) holds.}

\

By \eqref{exp}, for each small $r$, we can find a $\lambda_0^*$-homogeneous function $V_r$, such that
\begin{equation}\label{v-bi}
\left \|v- \sum b_i x_1 x_i - V_r \right \|_{L^\infty(B_r)} \le \delta(r) \, \|V_r\|_{L^\infty(B_r)},
\end{equation}
and, $$\delta(r) \to 0, \quad r^{\lambda_0^*-\alpha} \ge \|V_r\|_{L^\infty(B_r)} \ge r^{\lambda_0^*+\alpha} \quad \mbox{as $r \to 0$,} $$
for any fixed $\alpha>0$. We can define $p'$ as in Step 4 and obtain as in \eqref{u-p'} that
\begin{equation}\label{u-p2}
|u-p' - \eps V_r| \le 2 \eps \, \delta \, \|V_r\|_{L^\infty} + C_A'' \eps^{1+\alpha_F} r^2 \le 3  \eps \, \delta \, \|V_r\|_{L^\infty}
\end{equation}
in $B_{r}$. We choose $$\eps':= r^{-2}\|u-p'\|_{L^\infty(B_r)} \sim \eps r^{-2} \|V_r\|_{L^\infty(B_r)},$$ and $\delta=\delta(r)$ small, such that $\delta \le \eta/8$. The inequality above 
shows that conclusion (2) holds for all large $j$ and we reached a contradiction.

\end{proof} 

\begin{rem}\label{R4.6}
Notice that \eqref{u-p2} implies that $\|u-q\|_{L^\infty(B_r)} \ge c \eps' r^2$ for any quadratic polynomial $q$.
\end{rem}

\begin{rem}\label{R4.7}
In Steps 1-3 we expanded $v$ near the origin and used the uniqueness of the blow-ups when the frequency $\lambda_0 \in \{1,\frac 32, 2\}$. This is not necessary since, as in Steps 4-5, the expansion \eqref{exp} suffices. 
\end{rem}

Lemma \ref{L22} follows easily from the proof above and the results of Section \ref{Thin}.

\begin{proof}[Proof of Lemma \ref{L22}]
The class $\mathcal A (\eta, B_r)$ is invariant under multiplications by constants, and after relabeling $\eps$ we may assume that $\eps=2^{-2} \|u-p\|_{L^\infty(B_2)}$. 

Then $\hat u$ is $\eta$-approximated in $B_2$ by a $\lambda$-homogeneous function $W$ with $\lambda <3-\alpha_3$, and with $\|W\|_{L^\infty(B_2)} \sim 1$. 
Thus, in the proof of Lemma \ref{L2} above, the limiting function $v$ satisfies $$v \in \mathcal A (2 \eta, B_2), \quad \mbox{ and} \quad  v^e \in \mathcal A(4 \eta,B_2).$$
Since $\eta$ is chosen small we find $\lambda^*(x) \le 3-\alpha_3 < 3$ for all $x \in B_1$, and therefore Step 4 and case (1) cannot happen at any $x \in \Sigma^{d-1} 
\cap B_1$. This means that, as we iterate Lemma \ref{L2} at the origin, we always end up in case (2) and obtain a sequence of radii $r=r_n \to 0$ so that (see Remark \ref{R4.6})
$$ \|u-q\|_{B_r} \ge c \eps r^{2+\alpha_1}, \quad \quad  1/2 \ge r_{n+1}/r_n \ge c,$$ 
for any quadratic polynomial $q$. Hence, $0 \in \Sigma^{d-1}_a$ according to Definition \ref{AN2}, and part a) is proved.

Part b) is a direct consequence of Lemma \ref{cover}. Indeed, we can cover the set $\Gamma(v^e)$ of points $x \in B_1$ of frequency $\lambda^*(x) \in 
(2,3)$ by a finite cover $B_{r_i}(x_i)$ satisfying the desired properties and with $v^e \in \mathcal A(\eta/8, B_{2r_i}(x_i))$. 
This means that in each ball  $B_{2r_i}(x_i)$, $v$ satisfies \eqref{v-bi} with $\delta=\eta/8$, and by Step 5 we conclude that 
$$u \in \mathcal S_0 (p_i',\eps_i', 2 r_i), \quad (u-p_i')/\eps_i' \in \mathcal A (\eta, B_{2r_i}(x_i)).$$ 
 By Steps 1-3 of Lemma \ref{L2}, $\Sigma^{d-1}(u_j) \cap B_1$ belongs to the finite cover of $\Gamma(v^e)$ for all large $j$, 
 and the proof is complete.

\end{proof}

\end{document}